\documentclass[11pt,leqno]{article}
\usepackage{amsthm,amsfonts,amssymb,amsmath,oldgerm}
\usepackage{epsfig}
\numberwithin{equation}{section}
\usepackage[thinlines]{easybmat}

\renewcommand\a{\alpha}
\renewcommand\b{\beta}

\def\t{\tau}

\def\l{\lambda}
\def\eps{\varepsilon }

\renewcommand\a{\alpha}

\renewcommand\b{\beta}
\def\t1{{\widetilde{x}}}

\def\t{\tau}

\def\eps{\varepsilon}

\def\l{\lambda}
\newcommand\br{\begin{remark}}
\newcommand\er{\end{remark}}
\newcommand\bp{\begin{pmatrix}}
\newcommand\ep{\end{pmatrix}}
\newcommand\be{\begin{equation}}
\newcommand\ee{\end{equation}}
\newcommand\ba{\begin{equation}\begin{aligned}}
\newcommand\ea{\end{aligned}\end{equation}}
\newcommand\ds{\displaystyle}

\newcommand{\bap}{\begin{app}}
\newcommand{\eap}{\end{app}}
\newcommand{\begs}{\begin{exams}}
\newcommand{\eegs}{\end{exams}}
\newcommand{\beg}{\begin{example}}
\newcommand{\eeg}{\end{exaplem}}
\newcommand{\bpr}{\begin{proposition}}
\newcommand{\epr}{\end{proposition}}
\newcommand{\bt}{\begin{theorem}}
\newcommand{\et}{\end{theorem}}
\newcommand{\bc}{\begin{corollary}}
\newcommand{\ec}{\end{corollary}}
\newcommand{\bl}{\begin{lemma}}
\newcommand{\el}{\end{lemma}}
\newcommand{\bd}{\begin{definition}}
\newcommand{\ed}{\end{definition}}
\newcommand{\brs}{\begin{remarks}}
\newcommand{\ers}{\end{remarks}}

\newtheorem{theo}{Theorem}[section]

\newtheorem{exams}[theo]{Examples}

\numberwithin{equation}{section}
\newcommand{\mA}{{\mathbb A}}

\newcommand{\RR}{{\mathbb R}}

\newcommand{\ZZ}{{\mathbb Z}}

\newcommand{\CC}{{\mathbb C}}

\newtheorem{theorem}{Theorem}[section]
\newtheorem{proposition}[theorem]{Proposition}
\newtheorem{corollary}[theorem]{Corollary}
\newtheorem{lemma}[theorem]{Lemma}
\newtheorem{definition}[theorem]{Definition}

\newtheorem{example}[theorem]{Example}
\newtheorem{remark}[theorem]{Remark}

\newcommand{\beq}{\begin{equation}}
\newcommand{\eeq}{\end{equation}}
\newcommand{\bs}{\begin{split}}
\newcommand{\es}{\end{split}}

\newcommand{\vp}{\varphi}

\title{Pointwise stability estimates for periodic traveling wave solutions of systems of viscous conservation laws}

\author{\sc \small
Soyeun Jung\thanks{Indiana University, Bloomington, IN 47405;
soyjung@indiana.edu
 }}
\begin{document}

\maketitle

\begin{abstract}
In the previous paper \cite{J1},  we established pointwise bounds for the Green function of the linearized equation
associated with  spatially periodic traveling waves $\bar u$ of a system of reaction diffusion equations, and
also obtained pointwise nonlinear stability and behavior of $\bar u$ under small perturbations. In this paper, using periodic resolvent kernels and the Bloch-decomposition, we establish pointwise bounds for the Green function of the linearized equation
associated with periodic standing waves $\bar u$ of a system of conservation laws. We also show pointwise nonlinear stability of $\bar u$  by estimating decay of modulated perturbation $v$ of $\bar u$ under small perturbation $|v_0| \leq E_0(1+|x|)^{-\frac{3}{2}}$ for sufficiently small $E_0>0$.
\end{abstract}

%%%%%%%%%%%%%%%%%%%%%%%%%%%%%%%%%%%%%%%%%%%%%%%%%%%%%%%%%%%%%%%%%%%%%%%%%%%%%%%%%%%%%%%%%%%%%

%%%%%%%%%%%%%%%%%%%%%%%%%%%%%%%%%%%%%%%%%%%%%%%%

\section{Introduction}
%%%%%%%%%%%%%%%%%%%%%%%%%%%%%%%%%%%%%%%%%%%%%%%%

In this paper, we obtain pointwise bounds for the Green function of the linearized equations associated with spatially periodic traveling waves of systems of conservation laws extending previous work for reaction-diffusion systems in \cite{J1},  and using pointwise Green function bounds we establish the pointwise stability estimates for the periodic traveling waves. Compared with the previous work for reaction-diffusion systems, the main difference is that the Green function of the linearized operator with respect to the periodic traveling waves of conservation laws decays more slowly. This is because of the spectral structure of an eigenvalue $\l=0$ of the linear operator (Lemma \ref{nonsemisimple}).

We consider systems of viscous conservation laws of form
\be \label{cs}
u_t=u_{xx}+f(u)_x,
\ee
where $(x,t)\in \RR \times \RR^+,$  $u\in \mathcal{U}(\text{open}) \in  \RR^n$, and $f:\RR^n\rightarrow \RR^n$ is sufficiently smooth.

The $L^p$ nonlinear stability of the periodic traveling waves of systems of conservation laws have been obtained by Johnson-Zumbrun in all dimensions(\cite{JZ1} and \cite{JZ3}).  Here, following their basic approach, but a more detailed linear analysis, we establish the pointwise stability of the periodic traveling waves by deriving pointwise descriptions of localized modulated perturbations of $\bar u$.

\subsection{Assumptions}

We follow \cite{JZ1} and \cite{JZ3} in our assumptions. We assume the existence of an X-periodic traveling wave solution with boundary conditions $\bar u(0)=\bar u(X)=:\bar u_0$ of \eqref{cs} of the form
\be
u(x,t)=\bar u(x-st),
\notag
\ee
where $s$ is the speed of the traveling wave. Plugging $\bar u(x-st)$ into \eqref{cs}, we have
\be
-s\bar u'=\bar u''+f(\bar u)'.
\notag
\ee
Integrating both sides, we obtain the profile equation
\be \label{profile}
-s\bar u+q=\bar u' + f(\bar u),
\ee
where $(\bar u_0, q, s, X) \equiv$ constant. Without of loss of generality, we take s=0, that is, $\bar u(x)$ is a periodic standing wave solution of \eqref{cs}. For the existence of periodic solutions of \eqref{profile}, we make the following  assumptions (\cite{JZ1}, \cite{JZ3}, \cite{S}): \\
\\
(H1) The map $H: \RR \times  \mathcal{U} \times \RR \times \RR^n$ taking $(X;w,s,q) \mapsto u(X;w,s,q) - w$ is full rank at $(\bar X; \bar u(0), 0, \bar q)$, where $u(\cdot)$ is the solution operator of \eqref{profile}.  \\

By the Implicit Function Theorem, the condition (H1) implies that the set of periodic solutions of  \eqref{profile} vicinity of $\bar u$ form a smooth $(n+2)$-dimensional manifold $\{ \bar u^a(x-\a -s(a)t) \}$ with $\a \in \RR$ corresponding to translation and $a \in \RR^{n+1}$. \\

Linearizing  \eqref{cs} about a standing-waves solution
$\bar{u}(x)$ gives the second-order spectral  problem
\be\label{sp}
\begin{split}
\l v=Lv:
& =v_{xx}+(df(\bar u)v)_x \\
& = (\partial_x^2+df(\bar u)\partial_x + df(\bar u)_x)v
\end{split}
\ee
considered on the real Hilbert space $L^2(\RR)$. As coefficients of $L$ are $1$-periodic, Floquet theory implies that
the $L^2$ spectrum is purely continuous and corresponds to the union of
%the $L^\infty$ eigenvalues corresponding to considering the linearized operator with boundary conditions $v(x+1)=e^{i\xi}v(x)$ for all $x \in \RR$, where $\xi \in [-\pi, \pi)$ is referred to as the Floquet exponent and is uniquely defined mod $2\pi$.
%In particular, $\l \in \sigma(L)$ if and only if the spatially periodic spectral problem
$\l$ such that \eqref{sp} admits a bounded eigenfunction of the form
\be\label{spb}
v(x)=e^{i\xi x}w(x), \quad \xi \in \RR
\ee
where $w(x+1)=w(x)$, that is, the eigenvalues of the family
of associated Floquet, or Bloch, operators
\be\label{bloch operator}
L_\xi : = e^{-i\xi x} L e^{i \xi x}=(\partial_x^2+i\xi)^2+df(\bar u)(\partial_x+i\xi) + df(\bar u)_x, \quad \text{for} \quad \xi \in [-\pi, \pi),
\ee
considered as acting on $L^2$ periodic functions on $[0,1]$.
%determined by the defining relation
%\be
%L(e^{i\xi x}f)=e^{i\xi x}(L_{\xi}f) \quad \text{for}\quad  f \quad \text{periodic}.
%\ee
%The $L^2$ spectrum of the linearized operator $L$ is readily seen to be given by the union of the spectra of the Bloch operators.
%By continuity of the spectrum on the Floquet parameter $\xi$, and the discreteness of the spectrum of the elliptic operator L on the compact domain $[0,1]$, it follows that the spectra of L may be described as the union of countably many continuous surfaces $\l(\xi)$.

Recall that any function $g \in L^2(\RR)$ admits an inverse Bloch-Fourier representation
\be
g(x)= \frac{1}{2\pi}\int_{-\pi}^{\pi } e^{i\xi x}\check g(\xi,x) d\xi.
\notag
\ee
where $\check g(\xi,x) = \sum_{j \in \ZZ}e^{i2\pi jx}\hat g(\xi+2\pi j)$ is a $1$-periodic functions of $x$, and $\hat g(\cdot)$ denotes the Fourier transform of $g$ with respect to $x$. Indeed, using the Fourier transform we have
\be
 2\pi g(x)=\int_{-\infty}^\infty e^{i\xi x}\hat g(\xi) d\xi=\ds \sum_{j\in \ZZ} \int_{-\pi}^{\pi}e^{i(\xi+2\pi j)x}\hat g(\xi+2\pi j)d\xi = \int_{-\pi}^{\pi} e^{i\xi x}\check g(\xi,x) d\xi.
\notag
\ee
%where the summation and integral can be interchanged for
%Schwartz functions g.
Since
$L(e^{i\xi x}f)=e^{i\xi x}(L_{\xi}f)$
for $f$ periodic,
the Bloch-Fourier transform
diagonalizes the periodic-coefficient operator $L$, yielding the
inverse Bloch-Fourier transform representation
\be\label{inverse BF}
e^{Lt}g(x)= \frac{1}{2\pi}\int_{-\pi }^{\pi } e^{i\xi x}e^{L_\xi t}\check g(\xi,x)
d\xi.
\ee

We now discuss the strong spectral stability conditions of the periodic traveling waves $\bar u(\cdot)$.  By the translation invariant of \eqref{cs}, $\bar u'(x)$ is a 1-periodic function such that $L_0\bar u'=0$. It follows that $\l=0$ is an eigenvalue of the linear operator $L_0$. Moreover, the zero eigenspace of $L_0$ is at least $(n+1)$-dimensional (\cite{JZ1}, \cite{S}). Following \cite{JZ1} and \cite{OZ2}, we assume along with (H1) the following strong spectral stability conditions:\\
\\
 (D1) $\sigma(L) \subset \{Re\l < 0\} \cup \{ 0\}$. \\
 (D2) There exists a $\theta >0$ such that for all $\xi \in [-\pi,\pi]$ we have  $\sigma(L_{\xi}) \subset \{Re\l < -\theta|\xi|^2 \}$. \\
 (D3) $\l=0$ is an eigenvalue of $L_0$ of multiplicity exactly $n+1$. \\
\\
 Conditions (D1)-(D3) correspond to ``dissipativity" of the large-time behavior of the linearized system.  By standard spectral perturbation theory and assumption (D3), there exist $n+1$ smooth eigenvalues $\l_j(\xi)$ analytic at $\xi=0$ of $L_{\xi}$ bifurcating from $\l=0$ at $\xi=0$ with
\be \label{eigenvalue}
\l_j(\xi)=-ia_j\xi-b_j\xi^2+O(|\xi|^3),
\ee
where $a_j$ and $b_j >0$ are real. Moreover, we make the further nondegeneracy hypothesis(\cite{JZ1}, \cite{OZ2}): \\
\\
(H2) $a_j$ in \eqref{eigenvalue} are distinct.\\

\begin{remark}
In  (D3), $\l=0$ does not need to be a semisimple eigenvalue of $L_0$. This is the main difficulty of systems of conservation laws  compared with the pervious work for reaction-diffusion system(\cite{J1}).
\end{remark}

\begin{remark} [\cite{J1}]
The condition $(D3)$ may be readily verified by direct numerical Evans function analysis as described in \cite{BJNRZ1,BJNRZ2}.
\end{remark}

%%%%%%%%%%%%%%%%%%%%%%%%%%%%%%%%%%%%%%%%%%%%%%%%5
%%%%%%%%%%%%%%%%%%%%%%%%%%%%%%%%%%%%%%%%%%%%%%%%%
\subsection{First-order systems}

%%%%%%%%%%%%%%%%%%%%%%%%%%%%%%%%%%%%%%%%%%%%%%%%%%%
%%%%%%%%%%%%%%%%%%%%%%%%%%%%%%%%%%%%%%%%%%%%%%%%

Rewriting the eigenvalue equation $\eqref{sp}$ as a first-order system
\be \label{firstorder}
V'=\mA(\l,x)V,
\ee
where
\be
V=\bp v\\v'\ep, \quad \mA =\bp 0 & I \\ \lambda I -df(\bar u)_x  & - df(\bar u) \ep,
\notag
\ee
denote by $\mathcal F^{y \to x} \in \CC^{2n\times 2n}$ the solution operator of \eqref{firstorder}, defined by $\mathcal{F}^{y\to y}=I$, $\partial_x \mathcal{F}=\mA \mathcal{F}$.

By the definition of Bloch operators \eqref{bloch operator}, for each $\xi \in [-\pi, \pi]$, we have a second-order eigenvalue equation
\be\label{eig}
\l u = L_\xi u=u''-A_\xi u'-C_\xi u,
\ee
where $A_\xi = -2i\xi I-df(\bar u) \in \CC^{n\times n}$ and $C_\xi(x) = -df(\bar u)_x-i\xi df(\bar u)+\xi^2I \in \CC^{n\times n}$.
Rewriting \eqref{eig} as a first-order system
\be\label{firstorder sysem}
U'=\mA_\xi (x,\lambda)U,
\ee
where
\be\label{coeffs}
U=\bp u\\u'\ep, \quad \mA_\xi =\bp 0 & I \\ \lambda I + C_\xi & A_\xi \ep,
\ee
similarly, denote by $\mathcal{F}_\xi ^{y\to x}\in \CC^{2n\times 2n}$ the solution operator of \eqref{firstorder sysem}, defined by $\mathcal{F}_\xi ^{y\to y}=I$, $\partial_x
\mathcal{F}_\xi =\mA_\xi \mathcal{F}_\xi$.

\subsection{Spectral preparation}
We now state a key lemma from \cite{JZ1} describing the structure of the null space of the operator $L_{\xi}$ for sufficiently small $|\xi|$, which we will need in order to state our main result.  The condition (D3) tells that there are $n+1$ general eigenfunctions of $L_0$, and so the following lemma describes the way in which these eigenfunctions of $L_0$ bifurcate in $\xi$, see \cite{JZ1} for proof.

\begin{lemma}[JZ1]\label{nonsemisimple}
Assuming (H1)-(H2), (D1)- (D3), the eigenvalue $\l_j(\xi)$ of $L_{\xi}$ are analytic functions of $\xi$. Suppose further that 0 is a non-semisimple eigenvalue of $L_0$. Then the Jordan structure of the zero eigenspace of $L_0$ consists of an n-dimensional kernel and a single Jordan chain of height 2. In particular, $\bar u'$ spans the right eigendirection lying at the base of the Jordan chain while the left kernel of $L_0$ coincides with the n-dimensional subspace of constant functions. Moreover, for $|\xi|$ sufficiently small, there exist right and left eigenfunctions $q_j(\xi,x)$ and $\tilde q_j(\xi,x)$ of $L_{\xi}$ associated with $\l_j$ of form $q_j(\xi,x)=\sum_{k=1}^{n+1}\b_{j,k}(\xi)v_k(\xi,x)$ and $\tilde q_j(\xi,x)=\sum_{k=1}^{n+1}\tilde{\b}_{j,k}(\xi)\tilde v_k(\xi,x)$, where $\{v_j\}$ and $\{\tilde v_j\}$ are dual bases of the total eigenspace of $L_{\xi}$ associated with sufficiently small eigenvalues, analytic in $\xi$, with $\tilde v_j(0,x)$ constant for $j\neq n$ and $v_n(0,x)=\bar u'(x)$; $\tilde \b_{j,1}$, $\cdots$, $\tilde \b_{j,n-1}$, $\xi^{-1}\tilde \b_{j,n}$, $\tilde \b_{j,n+1}$ and $\b_{j,1}$, $\cdots$, $\b_{j,n-1}$, $\xi\b_{j,n}$, $\b_{j,n+1}$ are analytic in $\xi$; and $<\tilde q_j, q_k>=\delta_{j}^k$.
\end{lemma}

\begin{remark} \label{remark for spectral structure}
If $\l=0$ is a semisimple eigenvalue of $L_0$, the general right and left eigenvectors are genuine right and left eigenvectors. That is, we can simply say that there are right eigenfunctions $q_j(\xi,x)$ and left eigenfunctions $\tilde q_j(\xi,x)$ of the operator $L_{\xi}$, respectively, associated with the eigenvalue $\l_j(\xi)$ for each $j=1, \dots, n+1$, analytic in $\xi$ for sufficiently small $|\xi|$,  with the normalization condition  $<\tilde q_j, q_k>=\delta_{j}^k$. In the semisimple case, the pointwise Green function $G(x,t;y)$ bound of the linearized operator $L$ is similar to that of the previous work(reaction-diffusion case, \cite{J1}) with several modes of heat kernels, see \cite{JZ3} and \cite{OZ1} for the semisimple case.
\end{remark}

%%%%%%%%%%%%%%%%%%%%%%%%%%%%%%%%%%%%%%%%%%%%%%%%%%%%%
%%%%%%%%%%%%%%%%%%%%%%%%%%%%%%%%%%%%%%%%%%%%%%%
%%%%%%%%%%%%%%%%%%%%%%%%%%%%%%%%%%%%%%%%%%%%%%%

\subsection{Main results}
With these preparations, we state here our two main results. In theorem \ref{main theorem1}, we determine pointwise estimates for the  Green function $G(x,t;y)$ of \eqref{sp} which is  the linearization  about standing-wave solutions $\bar u$ of systems of conservation laws. In theorem \ref{main theorem2}, using pointwise bounds of G, we show pointwise stability estimates for $\bar u$ by deriving the pointwise decay of the modulated perturbation of $\bar u$ under the sufficiently small initial data.

\begin{theorem} \label{main theorem1}
The Green function $G(x,t;y)$ for equation \eqref{sp} satisfies the estimates:
\be
\begin{split}
G(x,t;y)
& = \bar u'(x)\sum_{j=1}^{n+1}\sum_{l \neq n}^{n+1} \check \b_{j,n}(0) \tilde \b_{j,l}(0) \tilde v_l(0,y)\text{errfn}\left(\frac{|x-y-a_jt|^2}{\sqrt{t}}\right) \\
& \qquad + \bar u'(x)\sum_{j=1}^{n+1} \check \b_{j,n}(0) \check{\tilde \b}_{j,n}(0) \tilde v_n(0,y)\frac{1}{\sqrt{4\pi b_jt}}e^{-\frac{|x-y-a_jt|^2}{4b_jt}} \\
%& = \bar u'(x)\sum_{j=1}^{n+1}\check \b_{j,n}(0) \left( \sum_{l \neq n}^{n+1}\tilde \b_{j,l}(0) \tilde v_l(0,y)\text{errfn}\left(\frac{|x-y-a_jt|^2}{\sqrt{t}}\right)+\check{\tilde \b}_{j,n}(0) \tilde v_n(0,y)\frac{1}{\sqrt{4\pi b_jt}}e^{-\frac{|x-y-a_jt|^2}{4b_jt}} \right) \\
& \qquad + O\left(\sum_{j=1}^{n+1} t^{-\frac{1}{2}}e^{-\frac{|x-y-a_j t|^2}{Mt}}\right), \\
G_y (x,t;y)
& = \bar u'(x)\sum_{j=1}^{n+1}\check \b_{j,n}(0)\left( \sum_{l \neq n}^{n+1} \tilde \b_{j,l}(0)  \tilde v_l(0,y)+ \check{\tilde \b}_{j,n}(0)\tilde v_n^\prime(0,y)\right) \frac{1}{\sqrt{4\pi b_jt}}e^{-\frac{|x-y-a_j t|^2}{4b_jt}}\\
& \qquad + O\left(\sum_{j=1}^{n+1} t^{-1}e^{-\frac{|x-y-a_j t|^2}{Mt}}\right), \\
\end{split}
\notag
\ee
uniformly on $t\geq 0$, for some sufficiently large constant $M>0$, where $\ds \check{\b}_{j,n}(0)= \lim_{\xi \rightarrow 0}\xi \b_{j,n}(\xi)$ and $\ds \check{\tilde \b}_{j,n}(0)= \lim_{\xi \rightarrow 0}\xi^{-1} \tilde \b_{j,n}(\xi)$ for $\b_{j,n}(\xi)$,  $\tilde \b_{j,n}(\xi)$, $v(\xi,x)$ and $\tilde v(\xi,x)$ defined in Lemma \ref{nonsemisimple}.
\end{theorem}

\begin{theorem}\label{main theorem2}
Let $\bar u$ be a periodic standing-wave solution of \eqref{cs} and let $u:=\tilde u-\bar u$, where $\tilde u$ is any solution of \eqref{cs} such that $|\tilde u(x,0)-\bar u(x,0)| \leq E_0(1+|x|)^{-\frac{3}{2}}$, $E_0$ sufficiently small. Then for some $\vp(\cdot,t) \in W^{2,\infty}$, we have the pointwise estimates
\be 
|\tilde u(x-\vp(x,t),t)-\bar u(x)| \leq CE_0(\theta + \psi_1+\psi_2),
\notag
\ee
where
\be
\theta(x,t):=\sum_{j=1}^{n+1}(1+t)^{-\frac{1}{2}}e^{-\frac{|x-a_jt|^2}{M't}},
\notag
\ee
\be
\psi_1(x,t):=\chi(x,t)\sum_{j=1}^{n+1}(1+|x|+t)^{-\frac{1}{2}}(1+|x-a_jt|)^{-\frac{1}{2}}
\notag
\ee
and
\be
\psi_2(x,t):=(1-\chi(x,t))(1+|x-a_1t|+\sqrt t)^{-\frac{3}{2}}+(1-\chi(x,t))(1+|x-a_{n+1}t|+\sqrt t)^{-\frac{3}{2}},
\notag
\ee
where $\chi(x,t) =1$ for $x \in [a_1t,a_{n+1}t]$ and zero otherwise, and $M'>0$ is a sufficiently large constant with $M'>M$.
\end{theorem}

%%%%%%%%%%%%%%%%%%%%%%%%%%%%%%%%%%%%%%%%%%%%%%%%
%%%%%%%%%%%%%%%%%%%%%%%%%%%%%%%%%%%%%%%%%%%%%%%%

%%%%%%%%%%%%%%%%%%%%%%%%%%%%%%%%%%%%%%%%%%%%%%%%
%%%%%%%%%%%%%%%%%%%%%%%%%%%%%%%%%%%%%%%%%%%%%%%%

\subsection{Discussion and open problems}

%%%%%%%%%%%%%%%%%%%%%%%%%%%%%%%%%%%%%%%%%%%%%%%%
%%%%%%%%%%%%%%%%%%%%%%%%%%%%%%%%%%%%%%%%%%%%%%%%

$L^p$ bounds  on the Green function of $L$ and $L^p$ stability have been obtained by Johnson and Zumbrun. We emphasize again that it is the pointwise description that is the main new aspect here. Pointwise Green function bounds for systems of viscous conservation laws have been obtained by Oh and Zumbrun(\cite{OZ1}) previously. However, this analysis was only for the nongeneric case in the conservation laws setting for in somewhat less detail which $\l=0$ is a semisimple eigenvalue of $L_0$. As mentioned in Remark \ref{remark for spectral structure}, if $\l=0$ is a semisimple eigenvalue of $L_0$, we can easily define the existence of the right and left eigenfuctions $q(\xi,x)$ and $\tilde q(\xi,x)$ of $L_{\xi}$ analytically  for sufficiently small $\xi$. In this case, we have the same Green function bounds as in the previous work for the reaction-diffusion case, only with several modes of heat kernels. However, for the generic case,  noting first that  the right and left eigenfuctions $q(\xi,x)$ and $\tilde q(\xi,x)$ of $L_{\xi}$ have more complicated descriptions as in Lemma \ref{nonsemisimple}, the Green function decays more slowly (theorem \ref{main theorem1}) than in the previous work, \cite{J1}.

Similarly to the previous work for reaction diffusion-systems, the key to the pointwise nonlinear analysis is to subtract out the first two terms of $G$ in Theorem \ref{main theorem1} from the integral representation of modulated perturbations  $v(x,t):=\tilde u(x-\vp(x,t),t)-\bar u(x)$ by defining $\vp(x,t)$ appropriately with an assumption $\vp(x,0)=0$, that is,  localized modulations (section \ref{iteration scheme}). However, the pointwise nonlinear analysis with nonlocalized modulations $h(x):=\vp(x,0)$, $|\partial_x h(x)|$, treated at $L^q \rightarrow L^p$ level in \cite{JNRZ1, JNRZ2, JNRZ3}, is an interesting direction for further investigation for both systems of reaction-diffusion and conservation laws. The main new ingredient compared to the localized case will be a detailed estimation of $e^{Lt}(\bar u' h_0)$ in terms of $|\partial_x h_0|$. With further effort, we could also give a description of behavior for both locallized and nonlocalized parts.

In our way of estimating nonlinear interactions, we follow the strategy of \cite{HZ}. Full details of the scattering part of the \cite{HZ} argument given in the more restricted situation considered here help clarify that argument as well. With further effort, one should be able to derive a more detailed description in terms of "nonlinear diffusion waves" as in \cite{HRZ}, by combining our argument with that of \cite{JNRZ3}.

%%%%%%%%%%%%%%%%%%%%%%%%%%%%%%%%%%%%%%%%%%%%%%%%%%
%%%%%%%%%%%%%%%%%%%%%%%%%%%%%%%%%%%%%%%%%%%%%%%

\section{The resolvent kernel} \label{resolvent kernel}

%%%%%%%%%%%%%%%%%%%%%%%%%%%%%%%%%%%%%%%%%%%%%%%%%%
%%%%%%%%%%%%%%%%%%%%%%%%%%%%%%%%%%%%%%%%%%%%%%%

In this section, we develop a formula for the resolvent kernel on the whole line and the periodic boundary conditions on $[0,1]$ using solution operators and projections.   Here,  `` whole-line "  means the kernel of periodic-coefficient operator considered as acting on $L^2(\mathbb R)$. For $\l$  in the resolvent set of $L$, we denote by $G_\l(x,y)$ the resolvent kernel defined by
\be
(L-\l I)G_\l(\cdot,y):=\delta_y \cdot I,
\notag
\ee
$\delta_y$ denoting the Dirac delta distribution centered at $y$.

We already constructed the formula for $\ds \bp G_{\xi,\lambda}\\ \partial_x G_{\xi,\lambda}\ep (x,y)$ in the previous paper, \cite{J1}. Here, we construct the formula of $\bp G_{\xi, \l}& \partial_y G_{\xi,\l} \ep(x,y)$. By [ZH](Lemma 4.3), we need to consider the adjoint operator $L_\xi^*$ of \eqref{firstorder sysem}, and $z=G_{\xi,\l}(x,\cdot)$  satisfies
\be\label{ad eig}
z\l = zL_\xi^* =z''-(zA_\xi)'-zC_\xi,
\ee
where $A_\xi = -2i\xi I-df(\bar u) \in \CC^{n\times n}$ and $C_\xi(x) = -df(\bar u)_x-i\xi df(\bar u)+\xi^2I \in \CC^{n\times n}$.
Rewriting \eqref{ad eig} as a first-order system
\be\label{ad firstorder sysem}
Z'=Z\tilde\mA_\xi (x,\lambda),
\ee
where
\be
U=\bp z & z'\ep, \quad \tilde\mA_\xi =\bp 0 & \l I-i\xi df(\bar u)+\xi^2I \\ I  &  2i\xi I+df(\bar u) \ep,
\notag
\ee
similarly, denote by $\tilde{\mathcal{F}}_\xi ^{x\to y}\in \CC^{2n\times 2n}$ the solution operator of \eqref{ad firstorder sysem}, defined by $\tilde{\mathcal{F}}_\xi ^{x\to x}=I$, $\partial_y
\tilde{\mathcal{F}}_\xi = \tilde{\mathcal{F}}_\xi\tilde\mA_\xi$.

In subsection \ref{example}, we give a simple example for construction of  $\bp G_{\xi, \l}& \partial_y G_{\xi,\l} \ep(x,y)$.

\subsection{The whole line case}

We constructed $\ds \bp G_{\xi,\lambda}\\ \partial_x G_{\xi,\lambda}\ep (x,y)$ in the previous paper \cite{J1}. We state here again with $\bp G_{\xi, \l}& \partial_y G_{\xi,\l} \ep(x,y)$. 

\bl\label{wholelem}
For all $\xi \in [-\pi,\pi]$, the whole line kernel satisfies
\be
\bp \mathcal G_{\xi,\lambda}\\ \partial_x\mathcal G_{\xi,\lambda}\ep (x,y)=
\begin{cases}
\mathcal F_\xi^{y\to x} \Pi_\xi^+(y)\bp 0\\I\ep, & x>y,\\
-\mathcal F_\xi^{y\to x} \Pi_\xi^-(y)\bp 0\\I\ep, & x < y,\\
\end{cases}
\notag
\ee

\be
\bp \mathcal G_{\xi,\lambda} & \partial_y \mathcal G_{\xi,\lambda}\ep (x,y)=
\begin{cases}
-\bp 0 & I \ep \tilde{\Pi}_{\xi}^- (x) \tilde{\mathcal F}_\xi^{x \to y} , & x>y, \\
\bp 0 & I \ep \tilde{\Pi}_{\xi}^+ (x) \tilde{\mathcal F}_\xi^{x \to y}, & x>y,
\end{cases}
\notag
\ee

where $\Pi_\xi^\pm$ and $\tilde{\Pi}_\xi^\pm$are projections onto the manifolds of solutions decaying as $x\to \pm\infty$ and $y\to \pm\infty$, respectively.
\el

\begin{proof}
We must only check the jump condition $ \Big[ \bp \mathcal G_{\xi,\lambda}\\ \mathcal G_{\xi,\lambda}'\ep \Big]\Big|_y=\bp 0 \\ I\ep$ and $\Big[\bp \mathcal G_{\xi,\lambda} & \partial_y \mathcal G_{\xi,\lambda}\ep \Big]|_y=\bp 0 & -I \ep $ which follows from $\mathcal{F}_\xi^{y\to y}=I=\tilde{\mathcal{F}}_\xi^{y\to y}$ and $\Pi_\xi^++\Pi_\xi^-=I=\tilde{\Pi}_\xi^++\tilde{\Pi}_\xi^-$, and the fact that
$\mathcal G_{\xi,\l}(\cdot,y)$ and $\mathcal G_{\xi,\l}(x,\cdot)$ decay at $\pm\infty$, which is clear by inspection.
\end{proof}

\subsection{The periodic case}

\bl\label{perlem}
For all $\xi \in [-\pi,\pi]$, the periodic kernel satisfies
\be
\bp G_{\xi,\lambda}\\ \partial_xG_{\xi,\lambda}\ep(x,y)=
\begin{cases}
\mathcal{F}_\xi ^{y\to x} M_\xi ^+ (y)\bp 0\\I\ep, & x>y,\\
-\mathcal{F}_\xi ^{y\to x} M_\xi ^- (y)\bp 0\\I\ep, & x\le y,\\
\end{cases}
\notag
\ee
where $M_\xi ^+ (y)=(I-\mathcal{F}_\xi ^{y\to y+1})^{-1}$ and
$M_\xi ^- (y)=-(I-\mathcal{F}_\xi ^{y\to y+1})^{-1}\mathcal{F}_\xi ^{y\to y+1}$, \\
\be
\bp G_{\xi,\lambda} & \partial_yG_{\xi,\lambda}\ep(x,y)=
\begin{cases}
-\bp 0 & I \ep \tilde{M}_\xi ^- (x) \tilde{\mathcal{F}}_\xi ^{x\to y} , & x>y,\\
\bp 0 & I \ep  \tilde{M}_\xi ^+ (x) \tilde{\mathcal{F}}_\xi ^{x\to y}, & x < y,\\
\end{cases}
\notag
\ee
where $\tilde M_\xi ^+ (x)=(I-\tilde{\mathcal{F}}_\xi ^{x\to x+1})^{-1}$ and
$\tilde M_\xi ^- (x)=-\tilde{\mathcal{F}}_\xi ^{x\to x+1}(I-\tilde{\mathcal{F}}_\xi ^{x\to x+1})^{-1}$, \\
\el

\begin{proof}
We must check the jump condition $ \Big[ \bp
G_{\xi,\lambda}\\ \partial_x G_{\xi,\lambda}\ep \Big]\Big|_y=\bp 0 \\ I\ep $ and $[\bp G_{\xi,\lambda} & \partial_yG_{\xi,\lambda} \ep]|_y= \bp 0 & -I \ep $  which follows
from $\mathcal{F}_\xi ^{y\to y}=I$ and $M_\xi ^++M_\xi ^-=I$, and periodicity, $\bp
G_{\xi,\lambda}\\ \partial_xG_{\xi,\lambda}\ep(0,y)= \bp G_{\xi,\lambda}\\ \partial_xG_{\xi,\lambda}\ep(1,y).$
By periodicity of the solution operator, $\mathcal{F}_\xi ^{0\to
y}\mathcal{F}_\xi ^{y\to 1}=\mathcal{F}_\xi ^{1\to y+1}\mathcal{F}_\xi ^{y\to
1}=\mathcal{F}_\xi ^{y\to y+1}$.
By direct computation, we obtain $\mathcal{F}_\xi ^{y\to 1}(I-\mathcal{F}_\xi ^{y\to y+1})^{-1}
 =\mathcal{F}_\xi ^{y\to 0}(I-\mathcal{F}_\xi ^{y\to y+X})^{-1}\mathcal{F}_\xi ^{y\to y+1}$
which gives us $\bp G_{\xi,\lambda}\\ \partial_x G_{\xi,\lambda}\ep(0,y)= \bp
G_{\xi,\lambda}\\ \partial_x G_{\xi,\lambda}\ep(1,y)$. Similarly we argue the periodicity for $\bp G_{\xi, \l}& \partial_y G_{\xi,\l} \ep(x,y)$.
\end{proof}

\subsection{Example}\label{example}

Consider the constant-coefficient scalar case
\be
u_t+au_x=u_{xx}, \quad a>0 \quad \text{constant}.
\notag
\ee
This gives a eigenvalue equation for each $\xi \in [-\pi, \pi]$,
\be
u'' - (a- i2\xi)u' - (\xi^2+ia\xi)u = \l u
\notag
\ee
Rewriting as a first-order system
\be
U' = \mA_\xi (x,\l)U,
\notag
\ee
where
\be
U=\bp u \\ u' \ep, \quad \mA_\xi = \bp 0 & 1 \\ \l+\xi^2 + ia\xi  &  a-i2\xi \ep.
\notag
\ee

By a direct calculation we can find two eigenvalues of $\mA_\xi$,
\be
\mu_\pm = \ds \frac{a-i2\xi \pm \sqrt{a^2+4 \l}}{2},
\notag
\ee
which are solutions of the characteristic equation
\be
\mu^2-(a-i2\xi)\mu-\l-\xi^2-ia\xi = 0.
\notag
\ee
Then for $Re \l >0$,  we can assume $Re\mu_- < 0$ and $Re\mu_+ > 0$.

To find $\bp G_{\xi, \l}& \partial_y G_{\xi,\l} \ep(x,y)$, let's consider the equation for each $\xi \in [-\pi,\pi]$,
\be \label{ex adjoint}
z''+(a-i2\xi)z'-(\xi^2+ia\xi)z=\l z
\ee
Rewriting as a first-order system
\be \label{ex adjoint fos}
Z'=Z \tilde{\mA}_\xi (x,\l),
\ee
where
\be
Z=\bp z & z' \ep, \quad \tilde{\mA}_\xi =\ds \bp 0 & \l+\xi^2 + ia\xi \\   1&  -a+i2\xi \ep.
\notag
\ee
Then the matrix $z=G_{\xi,\l}(x,\cdot)$ satisfies \eqref{ex adjoint} and $\bp G_{\xi, \l}(x,\cdot) & \partial_y G_{\xi,\l}(x,\cdot) \ep$ satisfies \eqref{ex adjoint fos}. It is easily see that there are two eigenvalues of  $\tilde{\mA}_\xi (x,\l)$
\be
\tilde{\mu}_\pm = \ds \frac{-a+i2\xi \pm \sqrt{a^2+4 \l}}{2}=-\mu_\mp.
\notag
\ee
By the same calculation, we find $G_{\xi,\l}(x,y)$ and $\partial_y G_{\xi,\l}(x,y)$,
\be
G_{\xi,\l}(x,y)=
\begin{cases}
\frac{e^{\tilde\mu_-(y-x+1)}}{(\tilde\mu_--\tilde\mu_+)(1-e^{\tilde\mu_-})}
-\frac{e^{\tilde\mu_+(y-x+1)}}{(\tilde\mu_--\tilde\mu_+)(1-e^{\tilde\mu_+})}, & x>y,\\\\
\frac{e^{\tilde\mu_-(y-x)}}{(\tilde\mu_--\tilde\mu_+)(1-e^{\tilde\mu_-})}
-\frac{e^{\tilde\mu_+(y-x)}}{(\tilde\mu_--\tilde\mu_+)(1-e^{\tilde\mu_+})},& x<y,
\end{cases}
\notag
\ee
and
\be
\partial_yG_{\xi,\l}(x,y)=
\begin{cases}
\frac{\tilde\mu_-e^{\tilde\mu_-(y-x+1)}}{(\tilde\mu_--\tilde\mu_+)(1-e^{\tilde\mu_-})}
-\frac{\tilde\mu_+e^{\tilde\mu_+(y-x+1)}}{(\tilde\mu_--\tilde\mu_+)(1-e^{\tilde\mu_+})}, & x>y,\\\\
\frac{\tilde\mu_-e^{\tilde\mu_-(y-x)}}{(\tilde\mu_--\tilde\mu_+)(1-e^{\tilde\mu_-})}
-\frac{\tilde\mu_+e^{\tilde\mu_+(y-x)}}{(\tilde\mu_--\tilde\mu_+)(1-e^{\tilde\mu_+})},& x<y.
\end{cases}
\notag
\ee
The solution operator of \eqref{ex adjoint fos} is
\be
\tilde{\mathcal F}_\xi^{x \to y}=e^{\tilde{\mA}_\xi(y-x)}=\tilde\Pi^+_\xi(x)e^{\tilde\mu_-(y-x)}+\tilde\Pi^-_\xi(x)e^{\tilde\mu_+(y-x)},
\notag
\ee
where
\be
\tilde\Pi^+_\xi
= \bp \frac{-\tilde\mu_+}{\tilde\mu_--\tilde\mu_+} &
\frac{-\tilde\mu_-\tilde\mu_+}{\tilde\mu_--\tilde\mu_+} \\
\frac{1}{\tilde\mu_--\tilde\mu_+} &
\frac{\tilde\mu_-}{\tilde\mu_--\tilde\mu_+} \ep
\quad \text{and} \quad
\tilde\Pi^-_\xi
= \bp \frac{\tilde\mu_-}{\tilde\mu_--\tilde\mu_+} &
\frac{\tilde\mu_-\tilde\mu_+}{\tilde\mu_--\tilde\mu_+} \\
\frac{-1}{\tilde\mu_--\tilde\mu_+} &
\frac{-\tilde\mu_+}{\tilde\mu_--\tilde\mu_+} \ep,
\notag
\ee
which satisfies
\be
\bp G_{\xi,\lambda} & \partial_yG_{\xi,\lambda}\ep(x,y)=
\begin{cases}
-\bp 0 & I \ep \tilde{M}_\xi ^- (x) \tilde{\mathcal{F}}_\xi ^{x\to y} , & x>y,\\
\bp 0 & I \ep  \tilde{M}_\xi ^+ (x) \tilde{\mathcal{F}}_\xi ^{x\to y}, & x < y,\\
\end{cases}
\notag
\ee
where $\tilde M_\xi ^+ (x)=(I-\tilde{\mathcal{F}}_\xi ^{x\to x+1})^{-1}$ and
$\tilde M_\xi ^- (x)=-\tilde{\mathcal{F}}_\xi ^{x\to x+1}(I-\tilde{\mathcal{F}}_\xi ^{x\to x+1})^{-1}$. \\

%%%%%%%%%%%%%%%%%%%%%%%%%%%%%%%%%%%%%%%%%%%%%%%%

%\subsection{Laplacian case}

%%%%%%%%%%%%%%%%%%%%%%%%%%%%%%%%%%%%%%%%%%%%%%%%%%

%For the Laplacian case $A=C=0$, the solution operator is \be
%\label{} \mathcal{F}^{y\to x}= e^{\mA(x-y)}=e^{-\sqrt{\l}(x-y)}\Pi_+
%+ e^{\sqrt{\l}(x-y)}\Pi_-, \ee where $\Pi_\pm$ are projections onto
%the manifolds of solutions decaying as $x\to \pm\infty$. For the
%Laplacian case, it is easy to get \be \label{} \displaystyle \Pi_\pm
%=\frac{1}{2} \bp I & \mp \sqrt{\l}^{-1}I \\
%\mp\sqrt{\l}I & I \ep. \ee

%\bl\label{Lapwholelem} The whole line kernel for the Laplacian case
%$A=C=0$ satisfies \ba\label{}  \bp G_\lambda\\G_\lambda'\ep=
%\begin{cases}
%\mathcal{F}^{y\to x} \Pi_+\bp 0\\I\ep, & x>y,\\
%-\mathcal{F}^{y\to x} \Pi_-\bp 0\\I\ep, & x\le y,\\
%\end{cases} \\ =
%\begin{cases}
%e^{-\sqrt{\l}(x-y)}\Pi_+\bp 0\\I\ep, & x>y,\\
%e^{\sqrt{\l}(x-y)}\Pi_-\bp 0\\I\ep, & x\le y,\\
%\end{cases} \ea
%\el

%%%%%%%%%%%%%%%%%%%%%%%%%%%%%%%%%%%%%%%%%

\section{Pointwise bounds on $G_{\xi,\l}(x,y)$ and $\partial_y G_{\xi,\l}(x,y)$ for $|\l|>R$, $R$ sufficiently large } \label{pointwse bounds of resolvent kernel}

%%%%%%%%%%%%%%%%%%%%%%%%%%%%%%%%%%%%%%%%%%%
In this section, we derive pointwise bounds on $G_{\xi,\l}(x,y)$ and $\partial_y G_{\xi,\l}(x,y)$. For the case $|\l| \leq R$, since $G_{\xi,\l}(x,y)$ is analytic in $\l$, we have
\be \label{compact resolvent}
|G_{\xi,\l}(x,y)|, |\partial_y G_{\xi,\l}(x,y)| \leq Ce^{-\theta |x-y|},
\ee
for all $x$,$y$, where $C$, $\theta >0$ are constants(See \cite{OZ1} and \cite{ZH}). For $|\l|>R$, $R$ sufficiently large, we use the direct construction of $G_{\xi,\l}(x,y)$ in Section \ref{resolvent kernel}. However, the argument for this part is exactly same as our previous work \cite{J1} for reaction-diffusion waves. Thus, we just state the pointwise bounds on $G_{\xi,\l}(x,y)$ and $\partial_y G_{\xi,\l}(x,y)$ for $|\l|>R$, $R$ sufficiently large without proof.
\begin{proposition}\label{proposition from J1}(\cite{J1})
For any $|\xi| \leq \pi$ and any $0 \leq x,  y \leq 1$,
\be
\begin{split}
|G_{\xi,\l}(x,y)| & \leq C|\l^{-1/2}|(e^{-\beta^{-1/2}|\l^{1/2}||x-y|}+e^{-\beta^{-1/2}|\l^{1/2}|(1-|x-y|)}) \\
|(\partial/\partial_x)G_{\xi,\l}(x,y)| & \leq C(e^{-\beta^{-1/2}|\l^{1/2}||x-y|}+e^{-\beta^{-1/2}|\l^{1/2}|(1-|x-y|)})
\end{split}
\notag
\ee
provided $|\l|$ is sufficiently large and $C>0$, that is, $|G_{\xi,\l}|$ is uniformly bounded as $|\l| \to \infty$. Here, $\b^{-1/2} \sim \ds \min_{\l \in \Omega \cap \{ |\l| >R\}} Re \sqrt{\l/|\l|}$. 
\end{proposition}

%%%%%%%%%%%%%%%%%%%%%%%%%%%%%%%%%%%%%%%%%%%%%%%%%%%%%%%%%%%%%%%%%%%%%%%%%%%%%%%%%%%%%%%%%%%

\section{Pointwise bounds on $G$} \label{pointwise bounds of green function}

%%%%%%%%%%%%%%%%%%%%%%%%%%%%%%%%%%%%%%%%%%%%%%%%%%%%%%%%%%%%%%%%%%%%%%%%%%%%%%%%%%%%%%%%%%%

We now prove Theorem \ref{main theorem1} which is pointwise bounds on the Green function $G(x,t;y)$ of the linear operator $L$ in \eqref{sp}. Let's define the sector
\be
\Omega : = \{ \l : Re(\l) \geq \theta_1 - \theta_2|Im(\l)| \},
\notag
\ee
where $\theta_1$, $\theta_2 >0$ are small constants.

We first state the the standard spectral resolution(inver Laplace transform) formula(see, \cite{ZH, OZ1}). We use this formula to prove Theorem \ref{main theorem1}. This is the reason we constructed the resolvent kernels and their bounds in the  previous sections. 

\begin{proposition}\label{proposition from ZH}(\cite{ZH})
The parabolic operator $\partial_t - L$ has a Green function $G(x,t;y)$ for each fixed $y$ and $(x,t) \neq (y,0)$ given by
\be \label{proposition}
G(x,t;y)=\frac{1}{2\pi i}\int_{\Gamma:=\partial(\Omega \backslash B(0,R))}e^{\l t}G_\l(x,y) d\l
\ee
for $R >0$ sufficiently large and $\theta_1$, $\theta_2 > 0$ sufficiently small. 
\end{proposition}

%\noindent
%The standard spectral resolution formula \eqref{proposition}, together with high frequency periodic resolvent bounds given previous section, will be the starting point for the proof of Theorem \ref{main theorem1}.
\begin{proof}[\textbf{Proof of Theorem \ref{main theorem1}.}]
\textbf{Case(i). $\ds\frac{|x-y|}{t}$ large.} We first consider the case that $|x-y|/t \geq S$, $S$ sufficiently large. For this case, as I mentioned in the previous paper \cite{J1}, it is
%CHANGED
hard to estimate $G$ through $|[G_\xi(x,t;y)]|$ directly, because of the problem of aliasing (\cite{J1}).
Instead we estimate $|G_\l(x,y)|$ first and we estimate $|G(x,t;y)|$ by \eqref{proposition}. This is treated by exactly the same argument as in \cite{ZH}. By \cite{ZH}, notice that
%hard to estimate $|[G_\xi(x,t;y)]|$ directly because of the problem of aliasing. Instead we estimate $|G_\l(x,y)|$ first and we estimate $|G(x,t;y)|$ by \eqref{proposition}. This is treated by exactly the same argument as in \cite{ZH}. By \cite{ZH}, notice that
\be
| G_\l (x,y) | \leq C |\l^{-1/2}|e^{-\b^{-1/2}|\l^{1/2}||x-y|},
\notag
\ee
for all $\l \in \Omega \backslash B(0,R)$ and $R>0$ sufficiently large,
and here, $\b^{-1/2} \sim \ds \min_{\l \in \Omega \cap \{ |\l| >R\}} Re \sqrt{\l/|\l|}$.

Finally we have
\be
|G(x,t;y)| \leq  C\Big| \int_{\Gamma}e^{\l t}G_\l(x,y)d\l \Big| \leq  t^{-\frac{1}{2}}e^{-\eta t}e^{-\frac{|x-y|^2}{8\b t}} \leq  t^{-\frac{1}{2}}e^{-\eta t}e^{-\frac{|x-y-a_jt|^2}{Mt}},
\notag
\ee
for all $a_j$, and for some $\eta >0$ and $M>0$ sufficiently large.(See \cite{ZH} for details) Here, the last inequality is from that $\ds\frac{|x-y|}{t}$ large.

%\begin{remark} \label{remark 4.1}
%\be \label{4.1}
%\begin{split}
%G_\l (x,y)
%& = \int_{-\pi}^{\pi}e^{i\xi x}e^{(L_\xi-\l I)t}\hat {\delta_y}(\xi,x) d\xi \\
%& = \int_{-\pi}^{\pi}e^{i\xi x}e^{(L_\xi-\l I)t}e^{-i\xi y}[\delta_y(x)] d\xi \\
%& = \int_{-\pi}^{\pi}e^{i\xi(x-y)}[G_{\xi,\l}(x,y)] d\xi \\
%& = \int_{-\pi}^{\pi}e^{i\xi(x-y)}\sum_{j \in \ZZ}\mathcal G_{\xi,\l}(x,y+j) d\xi \\
%& =  \int_{-\pi}^{\pi}\sum_{j \in \ZZ} e^{i\xi(x-y)}\mathcal G_{\xi,\l}(x,y+j) d\xi \\
%& = \int_{-\pi}^{\pi}\sum_{j \in \ZZ} e^{i\xi(x-y)}e^{-i\xi(x-y-j)}\mathcal G_\l(x,y+j) d\xi \\
%& =\int_{-\pi}^{\pi}\sum_{j \in \ZZ}e^{i\xi j} \mathcal G_\l(x,y+j)  d\xi \\
%& = \mathcal G_\l(x,y),
%\end{split}
%\ee
%here, the brackets $[\cdot]$ denotes the periodic extensions in x and y and the last equality is because of $\ds %\int_{-\pi}^{\pi} e^{-i\xi j} d\xi = 0$ for $j \neq 0$.
%\end{remark}

\bigskip

\textbf{Case (ii)}.\textbf{ $\ds\frac{|x-y|}{t} < S$ bounded.}
% I used [JZ2] for this first passage.
To begin, notice that by standard spectral perturbation theory [K], the total eigenprojection $P(\xi)$ onto the eigenspace of $L_\xi$ associated with the eigenvalues $\l(\xi)$  bifurcating from the $(\xi, \l(\xi))=(0,0)$ state is well defined and analytic in $\xi$ for $\xi$ sufficiently small, since the discreteness of the spectrum of $L_\xi$ implies that the eigenvalue $\l(\xi)$ is separated at $\xi=0$ from the remainder of the spectrum of $L_0$. By (D2), there exists an $\eps > 0$ such that $Re \sigma(L_\xi) \leq -\theta |\xi|^2$ for $0<|\xi|<2 \eps$. With this choice of $\eps$, we first introduce a smooth cut off function $\phi(\xi)$ such that
\be
\phi(\xi)=
\begin{cases}
1, & \text{if} \quad |\xi| \leq \varepsilon \\
0, & \text{if} \quad |\xi| \geq 2\varepsilon, \\
\end{cases}
\notag
\ee
where $\eps > 0$ is a sufficiently small parameter. Now from the inverse Bloch-Fourier transform representation, we split the Green function
\be
G(x,t;y)= \frac{1}{2\pi}\int_{-\pi}^{\pi}e^{i \xi x}e^{L_\xi t} \check {\delta_y}(\xi,x) d\xi
\notag
\ee
into its low-frequency part
\be
\mathcal L =  \frac{1}{2\pi}\int_{-\pi}^{\pi}e^{i \xi x}\phi(\xi) P(\xi)e^{L_\xi t} \check{\delta_y}(\xi,x) d\xi
\notag
\ee
and high frequency part
\be
\mathcal H = \frac{1}{2\pi}\int_{-\pi}^{\pi}e^{i \xi x}(1-\phi(\xi) P(\xi))e^{L_\xi t} \check{\delta_y}(\xi,x) d\xi.
\notag
\ee

Let's start by considering the second part $\mathcal H$. The proof of the high frequency part is similar to the previous work \cite{J1}.
Noting first that
\be
\check{\delta_y}(\xi,x)=\ds \sum_{j\in \ZZ}e^{i2\pi jx}\check{\delta_y}(\xi +2\pi j)=\ds \sum_{j\in \ZZ}e^{i2\pi jx}e^{-i(\xi+2\pi j)y}= e^{-i\xi y}\ds \sum_{j\in \ZZ}e^{i2\pi j(x-y)}=e^{-i \xi y}[\delta_y(x)],
\notag
\ee
we have for $|\xi| \geq 2\eps$, $\phi(\xi)=0$ and
\be
\begin{split}
& \int_{2\eps \leq |\xi| \leq \pi}e^{i\xi x}(1-\phi(\xi) P(\xi))e^{L_\xi t} \check{\delta_y}(\xi,x) d\xi \\
& = \int_{2\varepsilon \leq |\xi| \leq \pi }e^{i\xi x}e^{L_\xi t} \check{\delta_y}(\xi,x) d\xi \\
& = \int_{2\varepsilon \leq |\xi| \leq \pi }e^{i\xi (x-y)}e^{L_\xi t}[\delta_y(x)]d\xi \\
& = \int_{2\varepsilon \leq |\xi| \leq \pi }e^{i\xi (x-y)}[G_\xi(x,t;y)] d\xi,
\notag
\end{split}
\ee
where the brackets $[\cdot]$ denote the periodic extensions of the given function onto the whole line.
Assuming that $Re\sigma(L_\xi) \leq -\eta <0$ for $|\xi| \geq 2\eps$, we have
\be \label{4.12}
[G_\xi(x,t;y)]=\frac{1}{2\pi i}\int_{\Gamma_1} e^{\l t}[G_{\xi,\l}(x,y)] d\l,
\notag
\ee
here, we fix  $\Gamma_1 =\partial(\Omega \cup \{ Re\l \geq -\eta\})$ independent of $\xi$. Parameterizing $\Gamma_1$ by $Im\l : = k$, and applying the bounds of $\ds \sup_{|\xi| \leq \pi} |[ G_{\xi,\l}(x,y)] | < O(|\l^{-\frac{1}{2}}|)$ for large $|\l|$ in Proposition \ref{proposition from J1} and \eqref{compact resolvent}, we have
\be
\begin{split}
|[G_\xi(x,t;y)]|
& \leq C \int_{\Gamma_1} e^{Re \l t} |[G_{\xi,\l}(x,y)]| d\l \\
& \leq C e^{-\eta t } \int_{0}^{\infty} k^{-\frac{1}{2}}e^{- \theta_2 k t} dk \\
& \leq C t^{-\frac{1}{2}}e^{-\eta t} \\
& \leq C t^{-\frac{1}{2}}e^{-\frac{\eta }{2} t}\sum_{j=1}^{n+1}e^{-\frac{|x-y-a_jt|^2}{Mt}},
\end{split}
\notag
\ee
for large $M>0$. Here, the last inequality is from $\frac{|x-y-a_jt|}{t} < S_1$ bounded for all $j$. Indeed, for large $M>0$,
\be
e^{-\frac{|x-y-a_jt|^2}{Mt}} =e^{-(\frac{|x-y-a_jt|}{t})^2 \frac{t}{M}}  \geq e^{-\frac{S_1^2}{M}t} \geq e^{-\frac{\eta}{2}t},
\notag
\ee
and so,
\be\label{high}
\begin{split}
& \Big| \int_{2\eps \leq |\xi| \leq \pi}e^{i\xi x}(1-\phi(\xi) P(\xi))e^{L_\xi t} \check{\delta_y}(\xi,x) d\xi \Big| \\
& \leq C \ds \sup_{2\varepsilon \leq |\xi| \leq \pi}|[G_\xi(x,t;y)]| \\
& \leq C t^{-\frac{1}{2}}e^{-\frac{\eta }{2} t}\sum_{j=1}^{n+1}e^{-\frac{|x-y-a_jt|^2}{Mt}}.
\end{split}
\ee
For sufficiently small $|\xi|$, $I - \phi(\xi)P(\xi)=I-P(\xi)=Q(\xi)$, where $Q$ is the eigenprojection of $L_\xi$ associated with eigenvalues complementary to $\l_j(\xi)$ bifurcating from $\l=0$ at $\xi=0$, which have real parts strictly less than zero.
So we can estimate for $|\xi| \leq \eps$ in
the same way as in $\eqref{high}$. Combining these observations, we have
the estimate
\be
|B| \leq C t^{-\frac{1}{2}}e^{-\frac{\eta }{2} t}\sum_{j=1}^{n+1}e^{-\frac{|x-y-a_jt|^2}{Mt}},
\notag
\ee
for some $\eta > 0$ and sufficiently large $M>0$.

\bigskip

We now consider the low-frequency part  $\mathcal L$. By  Lemma \ref{nonsemisimple}, we know that $\xi \b_{j,n}(\xi)$ is analytic in $\xi$ for sufficiently small $|\xi|$. Letting $\ds\check{\b}_{j,n}(0)= \lim_{\xi \rightarrow 0}\xi \b_{j,n}(\xi)$ and $\check{\l}_j(\xi)=-ia_j\xi-b_j\xi^2$ we have
\be 
\begin{split}
\mathcal L
& = \frac{1}{2\pi}\int_{-\pi}^\pi e^{i \xi x}\phi(\xi) P(\xi)e^{L_\xi t} \check {\delta_y}(\xi,x) d\xi \\
& =\frac{1}{2\pi} \int_{-\pi}^\pi e^{i \xi (x-y)}\phi(\xi)\sum_{j=1}^{n+1}e^{\l_j(\xi)t}q_j(x,\xi)\tilde q_j(y,\xi) d\xi \\
& = \frac{1}{2\pi} \int_{-\pi}^\pi e^{i \xi (x-y)}\phi(\xi)\sum_{j,l,k=1}^{n+1}e^{\l_j(\xi)t} \b_{j,k}(\xi)v_k(x,\xi) \tilde \b_{j,l}(\xi) \tilde v_l(y,\xi) d\xi \\
& = \frac{1}{2\pi} \int_{-\pi}^\pi e^{i \xi (x-y)}\phi(\xi)\sum_{j,l \neq n}^{n+1}e^{\check \l_j(\xi)t} \check \b_{j,n}(0)v_n(0,x) \tilde \b_{j,l}(0) \tilde v_l(0,y)\xi^{-1} d\xi \\
& \quad + \frac{1}{2\pi} \int_{-\pi}^\pi e^{i \xi (x-y)}\phi(\xi)\sum_{j}^{n+1}e^{\l_j(\xi)t} (\xi\b_{j,n})(\xi)v_n(x,\xi) (\xi^{-1}\tilde \b_{j,n}(\xi)) \tilde v_n(y,\xi) d\xi \\
& \quad + \frac{1}{2\pi} \int_{-\pi}^\pi e^{i \xi (x-y)}\phi(\xi)\sum_{j,l \neq n}^{n+1} \left( e^{\l_j(\xi)t} \xi \b_{j,n}(\xi)v_n(x,\xi) \tilde \b_{j,l}(\xi) \tilde v_l(y,\xi) -e^{\check \l_j(\xi)t}\check \b_{j,n}(0)v_n(0,x) \tilde \b_{j,l}(0) \tilde v_l(0,y)   \right)\xi^{-1}d\xi \\
& \quad + \frac{1}{2\pi} \int_{-\pi}^\pi e^{i \xi (x-y)}\phi(\xi)\sum_{j,l\neq n, k\neq n}^{n+1}e^{\l_j(\xi)t} \b_{j,k}(\xi)v_k(x,\xi) \tilde \b_{j,l}(\xi) \tilde v_l(y,\xi) d\xi \\
& \quad + \frac{1}{2\pi} \int_{-\pi}^\pi e^{i \xi (x-y)}\phi(\xi)\sum_{j, k\neq n}^{n+1}e^{\l_j(\xi)t} \b_{j,k}(\xi)v_k(x,\xi) (\xi^{-1}\tilde \b_{j,n})(\xi) \tilde v_n(y,\xi)\xi d\xi \\
& = I+II+III+IV+V
\end{split}
\notag
\ee
We start with the estimate $I$.
\be
\begin{split}
I
&= \sum_{j,l \neq n}^{n+1} \check \b_{j,n}(0)v_n(0,x) \tilde \b_{j,l}(0) \tilde v_l(0,y)\frac{1}{2\pi} \int_{-\pi}^\pi e^{i \xi (x-y)}\phi(\xi)e^{\check \l_j(\xi)t}\xi^{-1} d\xi \\
& =  \sum_{j,l \neq n}^{n+1} \check \b_{j,n}(0)v_n(0,x) \tilde \b_{j,l}(0) \tilde v_l(0,y)\frac{1}{2\pi} p.v. \int_{-\infty}^\infty e^{i \xi (x-y)}e^{\check \l_j(\xi)t}\xi^{-1} d\xi \\
& \qquad + O\left( \int_{|\xi|\geq \varepsilon} e^{i \xi (x-y)}e^{\check \l_j(\xi)t}\xi^{-1} d\xi  \right) \\
& = \sum_{j,l \neq n}^{n+1} \check \b_{j,n}(0)v_n(0,x) \tilde \b_{j,l}(0) \tilde v_l(0,y)\text{errfn}\left(\frac{|x-y-a_jt|^2}{\sqrt{t}}\right) + O\left(\sum_{j=1}^{n+1} t^{-\frac{1}{2}}e^{-\frac{|x-y-a_j t|^2}{Mt}}\right).
\end{split}
\notag
\ee
See \cite{J1} for the detail estimate of   $O\left( \int_{|\xi|\geq \varepsilon} e^{i \xi (x-y)}e^{\check \l_j(\xi)t}\xi^{-1} d\xi  \right)$. Setting $\ds\check{\tilde \b}_{j,n}(0)= \lim_{\xi \rightarrow 0}\xi^{-1} \tilde \b_{j,n}(\xi)$, we separate $II$ into two parts,
\be
\begin{split}
II
& = \frac{1}{2\pi}\sum_{j=1}^{n+1} \int_{-\pi}^\pi e^{i \xi (x-y)}\phi(\xi)e^{\hat \l_j(\xi)t} \check{\b}_{j,n}(0)v_n(x,0) \check{\tilde \b}_{j,n}(0) \tilde v_n(y,0) d\xi \\
& \quad + \sum_{j=1}^{n+1}\frac{1}{2\pi} \int_{-\varepsilon}^\varepsilon e^{i \xi (x-y)}e^{\hat \l_j(\xi)t} \\
& \qquad \times \left(e^{O(|\xi|^3t)} (\xi\b_{j,n})(\xi)v_n(x,\xi) (\xi^{-1}\tilde \b_{j,n}(\xi)) \tilde v_n(y,\xi)-  \check{\b}_{j,n}(0)v_n(x,0) \check{\tilde \b}_{j,n}(0) \tilde v_n(y,0)\right) d\xi \\
& = \sum_{j=1}^{n+1}\check{\b}_{j,n}(0)v_n(x,0) \check{\tilde \b}_{j,n}(0) \tilde v_n(y,0)\frac{1}{2\pi}\int_{-\infty}^\infty e^{i \xi (x-y)}e^{\hat \l_j(\xi)t}d\xi +  O\left( \int_{|\xi|\geq \varepsilon} e^{i \xi (x-y)}e^{\check \l_j(\xi)t} d\xi  \right) \\
& \qquad + \sum_{j=1}^{n+1} \frac{1}{2\pi} \int_{-\varepsilon}^{\varepsilon} e^{i \xi (x-y)} e^{(-ia_j\xi-b_j\xi^2)t}\left(e^{O(\xi^3)t}-1+O(\xi)\right)d\xi \\
& = \sum_{j=1}^{n+1}\check{\b}_{j,n}(0)v_n(x,0) \check{\tilde \b}_{j,n}(0) \tilde v_n(y,0)\frac{1}{\sqrt{4\pi b_jt}}e^{-\frac{|x-y-a_jt|^2}{4b_jt}} + O\left( t^{-\frac{1}{2}} e^{-\frac{|x-y-a_jt|^2}{Mt}} \right)
\end{split}
\notag
\ee
Noting first that   $\xi \b_{j,n}(\xi)v_n(x,\xi) \tilde \b_{j,l}(\xi) \tilde v_l(y,\xi)$ is analytic in $\xi$, we have
\be
\begin{split}
III
& =  \frac{1}{2\pi} \int_{-\varepsilon}^{\varepsilon} e^{i \xi (x-y)}\sum_{j,l \neq n}^{n+1} \left( e^{\l_j(\xi)t} \xi \b_{j,n}(\xi)v_n(x,\xi) \tilde \b_{j,l}(\xi) \tilde v_l(y,\xi) -e^{\check \l_j(\xi)t}\check \b_{j,n}(0)v_n(0,x) \tilde \b_{j,l}(0) \tilde v_l(0,y)   \right)\xi^{-1}d\xi \\
& = \sum_{j,l \neq n}^{n+1} \frac{1}{2\pi} \int_{-\varepsilon}^{\varepsilon} e^{i \xi (x-y)} e^{(-ia_j\xi-b_j\xi^2)t}\left(e^{O(\xi^3)}-1+O(\xi)\right)\xi^{-1}d\xi \\
\end{split}
\notag
\ee
Similarly to \cite{J1}, viewing this as complex contour integral in complex variable $\xi$, define
\be
\a_j:=\frac{x-y-a_j}{2b_jt}
\notag
\ee
which is bounded because $|x-y|/t$ is bounded. Setting
\be
\tilde \a:=\min \{\varepsilon, \a_j\},
\notag
\ee
we have
\be \label{CI}
\begin{split}
|III|
& = \Big|\sum_{j=1}^{n+1} \int_{-\varepsilon}^{\varepsilon} e^{i \xi (x-y-a_jt)} e^{-b_j\xi^2t}\left(e^{O(\xi^3)t}-1+O(\xi)\right)\xi^{-1}d\xi \Big| \\
& = \Big|\sum_{j=1}^{n+1} \int_{-\varepsilon}^{\varepsilon} e^{i(\xi+i\tilde \a) (x-y-a_jt)} e^{-b_j(\xi+i\tilde \a) ^2t}\left(e^{O((\xi+i\tilde \a) ^3)t}-1+O(\xi+i\tilde \a) \right)(\xi+i\tilde \a) ^{-1}d\xi \Big| \\
& \qquad + \Big|\sum_{j=1}^{n+1} \int_{0}^{\tilde \a} e^{i(\varepsilon+iz) (x-y-a_jt)} e^{-b_j(\varepsilon+iz)^2t}\left(e^{O((\varepsilon+iz) ^3)t}-1+O(\varepsilon+iz) \right)(\epsilon+iz) ^{-1}dz \Big| \\
& \leq  C\sum_{j=1}^{n+1}e^{-b_jt\tilde \a^2}\int_{-\varepsilon}^{\varepsilon} e^{-b_j\xi^2t}\left(O(|\xi|^3t)+O(\xi)+O(\tilde \a) \right)|\xi+i\tilde \a |^{-1}d\xi \\
& \qquad + C\sum_{j=1}^{n+1}e^{-b_j\varepsilon^2t}\int_{0}^{\tilde \a}e^{-b_jz^2t} \left( O(|z|^3t)+O(\varepsilon)+O(z) \right)|\varepsilon+iz |^{-1}dz \\
& \leq  C\sum_{j=1}^{n+1}e^{-b_jt\tilde \a^2}\int_{-\varepsilon}^{\varepsilon} e^{-\frac{b_j}{2}\xi^2t}\left(O(\xi)+O(\tilde \a) \right)|\xi+i\tilde \a |^{-1}d\xi \\
& \qquad + C\sum_{j=1}^{n+1}e^{-b_j\varepsilon^2t}\int_{0}^{\tilde \a}e^{-\frac{b_j}{2}z^2t} \left( O(\varepsilon)+O(z) \right)|\varepsilon+iz |^{-1}dz \\
& \leq  C\sum_{j=1}^{n+1}e^{-b_jt\tilde \a^2}\int_{-\varepsilon}^{\varepsilon} e^{-\frac{b_j}{2}\xi^2t} d\xi +
 C\sum_{j=1}^{n+1}e^{-b_j\varepsilon^2t}\int_{0}^{\tilde \a}e^{-\frac{b_j}{2}z^2t} dz  \\
& \leq O\left( \sum_{j=1}^{n+1}t^{-\frac{1}{2}}e^{-\frac{|x-y-a_j t|^2}{Mt}}\right).
\end{split}
\ee
By Lemma \ref{nonsemisimple}, noting first that
\be
\b_{j,k}(\xi)v_k(x,\xi) \tilde \b_{j,l}(\xi) \tilde v_l(y,\xi) =O(1), \quad \text{for} \quad  l\neq n, k\neq n,
\notag
\ee
\be
\b_{j,k}(\xi)v_k(x,\xi) (\xi^{-1}\tilde \b_{j,n})(\xi) \tilde v_n(y,\xi) =O(1), \quad \text{for} \quad  k\neq n
\notag
\ee
and
\be
(\xi\b_{j,n})(\xi)v_n(x,\xi) (\xi^{-1}\tilde \b_{j,n}(\xi)) \tilde v_n(y,\xi) =O(1), 
\notag
\ee
we have
\be
\begin{split}
IV
& = \sum_{j, l \neq n, k \neq n }^{n+1} \int_{-\varepsilon}^\varepsilon e^{i \xi (x-y)}e^{\l_j(\xi)t} O(1)d\xi \\
& = \sum_{j, l \neq n, k \neq n }^{n+1} \int_{-\varepsilon}^\varepsilon e^{i \xi (x-y)}e^{(-ia_j\xi-b\xi^2)t}e^{O(\xi^3)t}d\xi \\
\end{split}
\notag
\ee
and
\be
\begin{split}
V
& = \sum_{j=1}^{n+1} \int_{-\varepsilon}^\varepsilon e^{i \xi (x-y)}e^{\l_j(\xi)t} O(\xi)d\xi.  \\
& =  \sum_{j=1}^{n+1}\int_{-\varepsilon}^\varepsilon e^{i \xi (x-y)}e^{(-ia_j\xi-b\xi^2)t}e^{O(\xi^3)t}O(\xi)d\xi. \\
\end{split}
\notag
\ee
Similarly to \eqref{CI}, we have $\ds IV=V=O\left( \sum_{j=1}^{n+1}t^{-\frac{1}{2}}e^{-\frac{|x-y-a_j t|^2}{Mt}}\right)$.

We now consider the estimate of $G_y(x,t;y)$. By Lemma \ref{nonsemisimple}, recalling $\tilde v_l(0,y)$ is constant for all $l \neq n$, we have
\be
\begin{split}
\partial_y I
& = \sum_{j,l \neq n}^{n+1} \check \b_{j,n}(0)v_n(0,x) \tilde \b_{j,l}(0) \tilde v_l(0,y)\frac{1}{2\pi} \int_{-\pi}^\pi e^{i \xi (x-y)}\phi(\xi)e^{\check \l_j(\xi)t}\xi^{-1}(-i\xi) d\xi \\
& = \sum_{j,l \neq n}^{n+1} \check \b_{j,n}(0)v_n(0,x) \tilde \b_{j,l}(0) \tilde v_l(0,y)\frac{1}{\sqrt{4\pi b_jt}}e^{-\frac{|x-y-a_jt|^2}{4b_jt}} +  O\left( \int_{|\xi|\geq \varepsilon} e^{i \xi (x-y)}e^{\check \l_j(\xi)t}d\xi  \right) \\
&  = \sum_{j,l \neq n}^{n+1} \check \b_{j,n}(0)v_n(0,x) \tilde \b_{j,l}(0) \tilde v_l(0,y)\frac{1}{\sqrt{4\pi b_jt}}e^{-\frac{|x-y-a_jt|^2}{4b_jt}} +  O\left( t^{-1}e^{-\frac{|x-y-a_jt|}{Mt}} \right), \\
\end{split}
\notag
\ee
and
\be
\begin{split}
\partial_y II
& = \sum_{j=1}^{n+1}\check{\b}_{j,n}(0)v_n(x,0) \check{\tilde \b}_{j,n}(0) \tilde v_n^\prime(y,0)\frac{1}{\sqrt{4\pi b_jt}}e^{-\frac{|x-y-a_jt|^2}{4b_jt}} + O\left( t^{-1} e^{-\frac{|x-y-a_jt|^2}{Mt}} \right).
\end{split}
\notag
\ee
Since $\tilde v_l(0,y)$ is constant for all $l \neq n$, $\partial_y \tilde v_l(\xi,y) = O(|\xi|)$ for all $l \neq n$, and so we have
\be
\begin{split}\
\partial_y III
& = \sum_{j,l \neq n}^{n+1} \frac{1}{2\pi} \int_{-\varepsilon}^{\varepsilon} e^{i \xi (x-y)} e^{(-ia_j\xi-b_j\xi^2)t}\left(e^{O(\xi^3)}-1+O(\xi)\right)O(1)d\xi \\
& \qquad +  \sum_{j,l \neq n}^{n+1} \frac{1}{2\pi} \int_{-\varepsilon}^{\varepsilon} e^{i \xi (x-y)} e^{(-ia_j\xi-b_j\xi^2)t}e^{O(\xi^3)t}O(\xi)d\xi \\
& =  O\left( \sum_{j=1}^{n+1}t^{-1}e^{-\frac{|x-y-a_jt|}{Mt}} \right).
\end{split}
\notag
\ee
Similarly, $\partial_y (\b_{j,k}(\xi)v_k(x,\xi) \tilde \b_{j,l}(\xi) \tilde v_l(y,\xi) )=O(|\xi|)$   for  all $l \neq n$, so we have
\be
\begin{split}
\partial_y IV
& =  \sum_{j,l \neq n, k \neq n}^{n+1} \frac{1}{2\pi} \int_{-\varepsilon}^{\varepsilon} e^{i \xi (x-y)} e^{(-ia_j\xi-b_j\xi^2)t}e^{O(\xi^3)t}O(\xi)d\xi \\
& =  O\left( \sum_{j=1}^{n+1}t^{-1}e^{-\frac{|x-y-a_jt|}{Mt}} \right).
\end{split}
\notag
\ee
Since $\partial_y (\b_{j,k}(\xi)v_k(x,\xi) (\xi^{-1}\tilde \b_{j,n})(\xi) \tilde v_n(y,\xi)) =O(1)$, we have
\be
\begin{split}
\partial_y V
=IV=  O\left( \sum_{j=1}^{n+1}t^{-1}e^{-\frac{|x-y-a_jt|}{Mt}} \right).
\end{split}
\notag
\ee

\end{proof}

%%%%%%%%%%%%%%%%%%%%%%%%%%%%%%%%%%%%%%%%%%%%%%%%%%%%%
%%%%%%%%%%%%%%%%%%%%%%%%%%%%%%%%%%%%%%%%%%%%%%%%%%%%%

\section{Pointwise description of perturbations of $\bar u$} \label{main behavior}

%%%%%%%%%%%%%%%%%%%%%%%%%%%%%%%%%%%%%%%%%%%%%%%%%%%%%%
%%%%%%%%%%%%%%%%%%%%%%%%%%%%%%%%%%%%%%%%%%%%%%%%%%%%%%

In this section we describe the pointwise bound of perturbations of \eqref{cs}. Let $\tilde u(x,t)$ be a solution of systems of conservation laws \eqref{cs} and let $\bar u(x)$ be a periodic stationary solution on $[0,1]$. We now define perturbations
\be \label{definition of v}
\begin{split}
u(x,t)  =\tilde u(x,t)-\bar u(x) \quad \text{and} \quad v(x,t)  =\tilde u(x-\vp(x,t),t)-\bar u(x),
\end{split}
\ee
for some unknown functions $\vp(x,t):\RR^2 \longrightarrow \RR$ to be determined later with $\vp(x,0)=0$.

In this section, using the pointwise estimate of the linear  operator $L$ in Theorem \ref{main theorem1}, we establish a pointwise description of perturbations $v$ for a initial condition $v_0=v(x,0)=u(x,0)$:
\be
|v_0(x)| \leq E_0(1+|x|)^{-\frac{3}{2}} \quad \text{and} \quad |v_0(x)|_{H^2} \leq E_0,
\notag
\ee
where $E_0>0$ sufficiently small. \\

Recalling Theorem \ref{main theorem1}, the Green function $G(x,t;y)$ of the linear equation $u_t=Lu$ satisfies the estimates:
\be
\begin{split}
G(x,t;y)
& = \bar u'(x)\sum_{j=1}^{n+1}\sum_{l \neq n}^{n+1} \check \b_{j,n}(0) \tilde \b_{j,l}(0) \tilde v_l(0,y)\text{errfn}\left(\frac{|x-y-a_jt|^2}{\sqrt{t}}\right) \\
& \qquad + \bar u'(x)\sum_{j=1}^{n+1} \check \b_{j,n}(0) \check{\tilde \b}_{j,n}(0) \tilde v_n(0,y)\frac{1}{\sqrt{4\pi b_jt}}e^{-\frac{|x-y-a_jt|^2}{4b_jt}} \\
& \qquad + O\left(\sum_{j=1}^{n+1} t^{-\frac{1}{2}}e^{-\frac{|x-y-a_j t|^2}{Mt}}\right), \\
G_y (x,t;y)
& = \bar u'(x)\sum_{j=1}^{n+1}\check \b_{j,n}(0)\left( \sum_{l \neq n}^{n+1} \tilde \b_{j,l}(0)  \tilde v_l(0,y)+ \check{\tilde \b}_{j,n}(0)\tilde v_n^\prime(0,y)\right) \frac{1}{\sqrt{4\pi b_jt}}e^{-\frac{|x-y-a_j t|^2}{4b_jt}}\\
& \qquad + O\left(\sum_{j=1}^{n+1} t^{-1}e^{-\frac{|x-y-a_j t|^2}{Mt}}\right), \\
\end{split}
\notag
\ee
uniformly on $t\geq 0$, for some sufficiently large constant $M>0$, where $\ds \check{\b}_{j,n}(0)= \lim_{\xi \rightarrow 0}\xi \b_{j,n}(\xi)$ and $\ds \check{\tilde \b}_{j,n}(0)= \lim_{\xi \rightarrow 0}\xi^{-1} \tilde \b_{j,n}(\xi)$ for $\b_{j,n}(\xi)$,  $\tilde \b_{j,n}(\xi)$, $v(\xi,x)$ and $\tilde v(\xi,x)$ defined in Lemma \ref{nonsemisimple}.

First off, let $\chi(t)$ be a smooth cut off function defined for $t \geq 0$ such that $\chi(t)=0$ for $0\leq t\leq 1$ and $\chi(t)=1$ for $t \geq 2$ and define
\be 
E(x,t;y):=\bar u'(x)e(x,t;y),
\notag
\ee
where
\be
\begin{split}
e(x,t;y)
& =\sum_{j=1}^{n+1}\sum_{l \neq n}^{n+1} \check \b_{j,n}(0) \tilde \b_{j,l}(0) \tilde v_l(0,y)\text{errfn}\left(\frac{|x-y-a_jt|^2}{\sqrt{t}}\right)\chi(t) \\
& \qquad + \sum_{j=1}^{n+1} \check \b_{j,n}(0) \check{\tilde \b}_{j,n}(0) \tilde v_n(0,y)\frac{1}{\sqrt{4\pi b_jt}}e^{-\frac{|x-y-a_jt|^2}{4b_jt}}\chi(t) \\
\end{split}
\notag
\ee
Now we set
\be
\begin{split}
\tilde G(x,t;y) = G(x,t;y)-E(x,t;y).
\notag
\end{split}
\ee
so that
\be \label{estimate of tilde G}
|\tilde G(x,t;y)| \leq Ct^{-\frac{1}{2}}\sum_{j=1}^{n+1}e^{-\frac{|x-y-a_jt|^2}{Mt}} \quad \text{and} \quad  |\tilde G_y(x,t;y)| \leq Ct^{-1}\sum_{j=1}^{n+1}e^{-\frac{|x-y-a_jt|^2}{Mt}}.
\ee
To establish a pointwise description of perturbations $v$, we first start with the nonlinear perturbation equation of $v$ (\cite{JZ1}).
\begin{lemma}[Nonlinear perturbation equations, \cite{JZ1}]
For $v$ defined in \eqref{definition of v}, we have
\be \label{nonlinear perturbation equation}
(\partial_t-L)v=-(\partial_t-L)\bar u'(x)\vp + Q_x+ R_x + (\partial_x^2+\partial_t)S,
\ee
where
\be \label{Q}
Q:=f(v(x,t)+\bar{u}(x))-f(\bar{u}(x))-df(\bar{u}(x))v=\mathcal{O}(|v|^2),
\ee
\be \label{R}
R:= -v\psi_t - v\psi_{xx}+ (\bar u_x +v_x)\frac{\vp_x^2}{1-\vp_x},
\ee
\be \label{S}
S:= v\vp_x =O(|v| |\vp_x|),
\ee
\end{lemma}

\begin{proof}
Direct computation; see \cite{JZ1}.
\end{proof}

%%%%%%%%%%%%%%%%%%%%%%%%%%%%%%%%%%%%%%%%%%%%%%%%%%%%%%
%%%%%%%%%%%%%%%%%%%%%%%%%%%%%%%%%%%%%%%%%%%%%%%%%%%%%

\subsection{Integral representation and $\vp$-evolution scheme } \label{iteration scheme}

%%%%%%%%%%%%%%%%%%%%%%%%%%%%%%%%%%%%%%%%%%%%%%%%%%%%%%
%%%%%%%%%%%%%%%%%%%%%%%%%%%%%%%%%%%%%%%%%%%%%%%%%%%%%

We now recall the nonlinear iteration scheme of \cite{JZ1}. Setting
\be 
N(x,t)=(Q_x+ R_x + (\partial_x^2+\partial_t)S)(x,t),
\notag
\ee
and applying Duhamel's principle to \eqref{nonlinear perturbation equation},  we obtain the integral representation of $v$
\be
\begin{split}
v(x,t)
& =-\bar u'(x)\vp(x,t) + \int_{-\infty}^\infty G(x,t;y) v_0(y)dy + \int_0^t \int_{-\infty}^\infty G(x,t-s;y)N(y,s)dyds.
\notag
\end{split}
\ee
for the nonlinear perturbation $v$.
Defining $\vp$ implicitly by
\be
\vp(x,t):=\int_{-\infty}^\infty e(x,t;y)v_0(y)dy +\int_0^t \int_{-\infty}^\infty e(x,t-s;y)N(y,s)dyds
\notag
\ee
to subtract out $E(x,t;y)$ from $G(x,t;y)$, we have the new integral representation of $v$
\be \label{integral representation of v}
v(x,t) =\int_{-\infty}^\infty \tilde G(x,t;y) v_0(y)dy +\int_0^t \int_{-\infty}^\infty \tilde G(x,t-s;y)N(y,s)dyds.
\ee
Differentiating and using $e(x,t;y)=0$ for $0<t \leq 1$ we obtain
\be \label{derivative of psi}
\partial_t^k \partial_x^m \vp(x,t) = \int_{-\infty}^\infty \partial_t^k \partial_x^m e(x,t;y)v_0 dy + \int_0^t \int_{-\infty}^\infty \partial_t^k \partial_x^m e(x,t-s;y)N(y,s)dyds.
\ee
Together, \eqref{integral representation of v} and \eqref{derivative of psi}
 form a  complete system in $(v,\partial_t^k\vp, \partial_x^m\vp)$, $0\leq  k \leq 1$, $0\leq m \leq 2$, that is, $v$ and derivatives of
$\vp$,
from solutions of which we may afterward recover the shift function
$\vp$ by integration in $x$, completing the description of $\tilde u$.

 %%%%%%%%%%%%%%%%%%%%%%%%%%%%%%%%%%%%%%%%%%%%%%%%%%%%%%
%%%%%%%%%%%%%%%%%%%%%%%%%%%%%%%%%%%%%%%%%%%%%%%%%%%%%

\subsection{Pointwise description of $v$ for initial perturbations $|v_0(x)| \leq E_0(1+|x|)^{-\frac{3}{2}}$ with $|v_0(x)|_{H^2} \leq E_0$, sufficiently small $E_0>0$ } \label{initial perturbation}.

%%%%%%%%%%%%%%%%%%%%%%%%%%%%%%%%%%%%%%%%%%%%%%%%%%%%%%
%%%%%%%%%%%%%%%%%%%%%%%%%%%%%%%%%%%%%%%%%%%%%%%%%%%%%

In this section, we prove Theorem \ref{main theorem2}. We start with $L^p$ estimates of $v$ , $u$ and $\vp$ which are proved in \cite{JZ1}. We state the main theorem of \cite{JZ1} describing the $L^p$ stability of periodic standing waves of \eqref{cs} in dimension $d=1$. We use the following Theorem \ref{nonlinear stability} when we derive pointwise estimates of the nonlinear terms of $v$ in \eqref{integral representation of v}, and this is the reason why we need $H^2$ condition in our initial perturbations.

\begin{theorem}[Nonlinear stability, \cite{JZ1}] \label{nonlinear stability} Let $v(x,t)$ and $u(x,t)$ be defined as in \eqref{definition of v} and $|u_0(x)|=|v_0(x)|_{L^1 \cap H^2(\RR)} \leq  E_0$, for  sufficiently small $E_0>0$. Then for all $t\geq 0$ and $p\geq 1$ we have the estimates
\be 
\begin{split}
&|v(\cdot,t)|_{L^p(\RR)}(t) \leq CE_0(1+t)^{-\frac{1}{2}(1-\frac{1}{p})} \\
&|u(\cdot,t)|_{L^p(\RR)}(t), \quad |\vp(\cdot,t)|_{L^p(\RR)}(t) \leq CE_0(1+t)^{-\frac{1}{2}(1-\frac{1}{p})+\frac{1}{2}} \\
&|v(\cdot,t)|_{H^2(\RR)}(t), \quad |(\vp_t, \vp_x)(\cdot,t)|_{H^2(\RR)}(t) \leq CE_0(1+t)^{-\frac{1}{4}}.
\end{split}
\notag
\ee
\end{theorem}

\begin{proof}
See \cite{JZ1} for the proof.
\end{proof}

To prove Theorem \ref{main theorem2}, we first prove the following lemma. We follow the strategy of \cite{HZ}. We give here details of \cite{HZ} to help clarify that argument.

%\begin{theorem}\label{main theorem2}
%Let $\bar u$ be a periodic standing-wave solution of \eqref{cs} and let $u:=\tilde u-\bar u$, where $\tilde u$ is any solution of \eqref{cs} such that $|\tilde u(x,0)-\bar u(x,0)| \leq E_0(1+|x|)^{-\frac{3}{2}}$, $E_0$ sufficiently small. Then for some $\vp(\cdot,t) \in W^{2,\infty}$, we have the pointwise estimates
%\be \label{estimate in main2}
%|\tilde u(x-\vp(x,t),t)-\bar u(x)| \leq CE_0(\theta + \psi_1+\psi_2),
%\ee
%where
%\be
%\theta(x,t):=\sum_{j=1}^{n+1}(1+t)^{-\frac{1}{2}}e^{-\frac{|x-a_jt|^2}{Mt}},
%\ee
%\be
%\psi_1(x,t):=\chi(x,t)\sum_{j=1}^{n+1}(1+|x|+t)^{-\frac{1}{2}}(1+|x-a_jt|)^{-\frac{1}{2}}
%\ee
%and
%\be
%\psi_2(x,t):=(1-\chi(x,t))(1+|x-a_1t|+\sqrt t)^{-\frac{3}{2}}+(1-\chi(x,t))(1+|x-a_{n+1}t|+\sqrt t)^{-\frac{3}{2}},
%\ee
%where $\chi(x,t) =1$ for $x \in [a_1t,a_{n+1}t]$ and zero otherwise, and $M>0$ is a sufficiently large constant.
%\end{theorem}
%
%

\begin{lemma} \label{main lemma}
Suppose  that the initial perturbation $v_0$ satisfies $|v_0(x)| \leq E_0(1+|x|)^{-\frac{3}{2}}$ and $|v_0(x)|_{H^2} \leq E_0$, for sufficiently small $E_0>0$. For $v$, $\vp_t$, $\vp_x$ and  $\vp_{xx}$ defined in the integral system \eqref{integral representation of v} - \eqref{derivative of psi}, define
\be \label{zeta}
\zeta(t):=\sup_{0\leq s\leq t, x\in \RR} |(v,\vp_t, \vp_x, \vp_{xx})(x,s)|(\theta + \psi_1+\psi_2)^{-1},
\ee
where
\be
\theta(x,t):=\sum_{j=1}^{n+1}(1+t)^{-\frac{1}{2}}e^{-\frac{|x-a_jt|^2}{M't}},
\notag
\ee
\be
\psi_1(x,t):=\chi(x,t)\sum_{j=1}^{n+1}(1+|x|+t)^{-\frac{1}{2}}(1+|x-a_jt|)^{-\frac{1}{2}}
\notag
\ee
and
\be
\psi_2(x,t):=(1-\chi(x,t))(1+|x-a_1t|+\sqrt t)^{-\frac{3}{2}}+(1-\chi(x,t))(1+|x-a_{n+1}t|+\sqrt t)^{-\frac{3}{2}},
\notag
\ee
where $\chi(x,t) =1$ for $x \in [a_1t,a_{n+1}t]$ and zero otherwise, and $M'>0$ is a sufficiently large constant with $M'>M$.
Then, for all $t\geq 0$ for which $\zeta(t)$ defined in \eqref{zeta} is finite,
\be \label{inequality of zeta}
\zeta(t) \leq C(E_0+\zeta^2(t))
\ee
for some constant $C>0$.
\end{lemma}

\begin{proof}
It is enough to estimate $v$,
\be \label{estimate of v}
|v(x,t)|  \leq C(E_0+\zeta^2(t))(\theta + \psi_1+\psi_2).
\ee
We can prove similarly for ($\vp_t$, $\vp_x$, $\vp_{xx}$) because 
\be
|\partial_t^k \partial_x^m e(x,t;y)| \lesssim |\tilde G(x,t;y)| \quad \text{and} \quad |\partial_y(\partial_t^k \partial_x^m e(x,t;y))| \lesssim |\partial_y\tilde G(x,t;y)|,
\notag
\ee
for $0\leq k \leq 1$ and $0 \leq m \leq 2$.
Notice first that by Theorem \ref{nonlinear stability}, we have $|v_x|_{\infty} \leq |v|_{H^2} \leq CE_0(1+t)^{-\frac{1}{4}} \leq C$. Then by \eqref{Q}-\eqref{S} and \eqref{zeta}, for all $y \in \RR$ and  $0 \leq s \leq t$,
\be
|(Q,R,S)(y,s)| \leq C|(v, \vp_s, \vp_y, \vp_{yy})(y,s)|^2 C\leq \zeta^2(t)(\theta + \psi_1+\psi_2)^{2},
\notag
\ee
and hence by \eqref{integral representation of v} and integration by parts, we have
\be
\begin{split}
|v(x,t)|
& \leq \int_{-\infty}^\infty |\tilde G(x,t;y)|| v_0(y)|dy +\int_0^t \int_{-\infty}^\infty |\tilde G_y(x,t-s;y)||(v, \vp_s, \vp_y, \vp_{yy})(y,s)|^2 dyds \\
& \leq \int_{-\infty}^\infty |\tilde G(x,t;y)|| v_0(y)|dy +\zeta^2(t) \int_0^t \int_{-\infty}^\infty |\tilde G_y(x,t-s;y)||\theta + \psi_1+\psi_2|^{2}dyds.  \\
\end{split}
\notag
\ee
To argue \eqref{estimate of v}, we need to prove following estimates:
\be \label{linear estimate}
\int_{-\infty}^\infty |\tilde G(x,t;y)||v_0(y)|dy \leq CE_0(\theta+\psi_1+\psi_2)(x,t),
\ee

\be \label{nonlinear estimate1}
\int_0^t \int_{-\infty}^\infty |\tilde G_y(x,t-s;y)||\theta(y,s)|^2dyds \leq C(\theta+\psi_1+\psi_2)(x,t),
\ee

\be \label{nonlinear estimate2}
\int_0^t \int_{-\infty}^\infty |\tilde G_y(x,t-s;y)||\psi_1(y,s)|^2dyds \leq C(\theta+\psi_1+\psi_2)(x,t),
\ee

\be \label{nonlinear estimate3}
\int_0^t \int_{-\infty}^\infty |\tilde G_y(x,t-s;y)||\psi_2(y,s)|^2dyds \leq C(\theta+\psi_1+\psi_2)(x,t).
\ee

%%%%%%%%%%%%%%%%%%%%%%%%%%%%%%%%%%%%%%%%%%%%%%%%%%%%%
%%%%%%%%%%%%%%%%%%%%%%%%%%%%%%%%%%%%%%%%%%%%%%%%%%

\begin{proof}[\textbf{Proof of the estimate (\ref{linear estimate})}] We start with  the linear estimate of $v$,
\be
\int_{-\infty}^\infty |\tilde G(x,t;y)||v_0(y)|dy \leq CE_0(\theta+\psi_1+\psi_2)(x,t).
\notag
\ee
By \cite{J1} and \cite{HZ}, we have
\be
\begin{split}
\int_{-\infty}^\infty |\tilde G(x,t;y)||v_0(y)|dy
& \leq CE_0\int_{-\infty}^\infty  t^{-\frac{1}{2}}\sum_{j=1}^{n+1}e^{-\frac{|x-y-a_jt|^2}{Mt}} (1+|y|)^{-\frac{3}{2}}dy  \\
& \leq CE_0\sum_{j=1}^{n+1}\Big[(1+|x-a_jt|+\sqrt t)^{-\frac{3}{2}}+(1+t)^{-\frac{1}{2}}e^{-\frac{|x-a_jt|^2}{M't}}\Big].
\end{split}
\notag
\ee
We now need to show that for any $j$,
\be
(1+|x-a_jt|+\sqrt t)^{-\frac{3}{2}} \leq C(\theta+\psi_1+\psi_2)(x,t).
\notag
\ee
We consider several cases. Here we assume $a_1 < a_2 < \cdots < a_{n+1}$.  \\

\textbf{case1.} $x \leq a_1t$ or $x\geq a_{n+1}t$.  For any $j=1, \cdots, n+1$, $|x-a_1t| \leq |x-a_jt|$ for $x\leq a_1t$ and $|x-a_{n+1}t|\leq |x-a_jt|$ for $x\geq a_{n+1}$. Thus
\be
(1+|x-a_jt|+\sqrt t)^{-\frac{3}{2}} \leq (1+|x-a_1t|+\sqrt t)^{-\frac{3}{2}} , \quad \text{for} \quad x\leq a_1t,
\notag
\ee
and
\be
(1+|x-a_jt|+\sqrt t)^{-\frac{3}{2}} \leq (1+|x-a_{n+1}t|+\sqrt t)^{-\frac{3}{2}} , \quad \text{for} \quad x\geq a_{n+1}t. 
\notag
\ee
\\

\textbf{case2.} $x \in [a_1t, a_{n+1}t]$,  and  $x$ and $a_j$ have opposite signs. In this case, $|x-a_jt| \leq C(|x|+t)$ because of no cancellation. So
\be
\begin{split}
(1+|x-a_jt|+\sqrt t)^{-\frac{3}{2}}
& =(1+|x-a_jt|)^{-\frac{1}{2}}(1+|x-a_jt|)^{-\frac{1}{2}} \\
& \leq C(1+|x|+t)^{-\frac{1}{2}}(1+|x-a_jt|)^{-\frac{1}{2}}
\end{split}
\notag
\ee
\\

\textbf{case3.} $x \in [a_1t, a_{n+1}t]$, and $x$ and $a_j$ have same signs. If $x \in [\frac{a_j}{2}t, 2a_jt]$, then $t^{-\frac{1}{2}} \leq C(|x|+t)^{-\frac{1}{2}}$, so
\be
\begin{split}
(1+|x-a_jt|+\sqrt t)^{-\frac{3}{2}}
& \leq C(1+|x-a_jt|)^{-\frac{1}{2}}(1+\sqrt t)^{-1} \\
& \leq C(1+|x-a_jt|)^{-\frac{1}{2}}(1+t)^{-\frac{1}{2}} \\
& \leq C(1+|x-a_jt|)^{-\frac{1}{2}}(1+t+|x|)^{-\frac{1}{2}} .
\end{split}
\notag
\ee
If $x \notin [\frac{a_j}{2}t, 2a_jt]$, there can be only limited cancellation between $x$ and $a_jt$, and so $|x-a_jt|\leq C(|x|+t)$, that is,
\be
\begin{split}
(1+|x-a_jt|+\sqrt t)^{-\frac{3}{2}}
& \leq C(1+|x-a_jt|+\sqrt t)^{-\frac{1}{2}}(1+|x-a_jt|)^{-\frac{1}{2}} \\
& \leq C(1+|x-a_jt|+\sqrt t)^{-\frac{1}{2}}(1+|x|+t)^{-\frac{1}{2}}
\end{split}
\notag
\ee
\end{proof}

%%%%%%%%%%%%%%%%%%%%%%%%%%%%%%%%%%%%%%%%%%%%%%%%%%%%%
%%%%%%%%%%%%%%%%%%%%%%%%%%%%%%%%%%%%%%%%%%%%%%%%%%

\begin{proof}[\textbf{Proof of the estimate (\ref{nonlinear estimate1})}]
We now estimate the first nonlinear term of $v$,
\be
I = \int_0^t \int_{-\infty}^\infty |\tilde G_y(x,t-s;y)||\theta(y,s)|^2 dyds \leq C(\theta+\psi_1+\psi_2)(x,t).
\notag
\ee
By \eqref{estimate of tilde G},
\be
\begin{split}
I
& \leq  \int_0^t \int_{-\infty}^\infty (t-s)^{-1}\sum_{j=1}^{n+1}e^{-\frac{|x-y-a_j(t-s)|^2}{M(t-s)}} \theta^2(y,s)dyds \\
%& = \int_0^t \int_{-\infty}^\infty (t-s)^{-1}\sum_{j=1}^{n+1}e^{-\frac{|x-y-a_j(t-s)|^2}{m(t-s)}}\big[\sum_{k=1}^{n+1}(1+s)^{-\frac{1}{2}}e^{-\frac{|y-a_ks|^2}{ms}}\big]^2 dyds \\
%& \leq  \int_0^t \int_{-\infty}^\infty (t-s)^{-1}\sum_{j=1}^{n+1}e^{-\frac{|x-y-a_j(t-s)|^2}{m(t-s)}}\sum_{k=1}^{n+1}(1+s)^{-1}e^{-\frac{|y-a_ks|^2}{ms}} dyds \\
%& = \int_0^t \int_{-\infty}^\infty (t-s)^{-1}(1+s)^{-1}\sum_{j=1}^{n+1}\sum_{k=1}^{n+1}e^{-\frac{|x-y-a_j(t-s)|^2}{m(t-s)}}e^{-\frac{|y-a_ks|^2}{ms}} dyds \\
& \leq C\int_0^t \int_{-\infty}^\infty (t-s)^{-1}(1+s)^{-1}\sum_{j,k=1}^{n+1}e^{-\frac{|x-y-a_j(t-s)|^2}{M(t-s)}}e^{-\frac{2|y-a_ks|^2}{M's}} dyds \\
& \leq C\int_0^t \int_{-\infty}^\infty (t-s)^{-1}(1+s)^{-1}\sum_{j,k=1}^{n+1}e^{-\frac{|x-y-a_j(t-s)|^2}{N(t-s)}}e^{-\frac{|y-a_ks|^2}{Ns}} dyds,
%& \qquad + C \int_0^t \int_{-\infty}^\infty(t-s)^{-1}(1+s)^{-1}\sum_{j \neq k}e^{-\frac{|x-y-a_j(t-s)|^2}{M(t-s)}}e^{-\frac{|y-a_ks|^2}{Ms}} dyds \\
%& = A+B
\end{split}
\notag
\ee
for sufficiently large $N>0$ with $ \frac{M'}{2} < N < M'$. Noting first that for any $j=k$,
\be
\int_{-\infty}^\infty e^{-\frac{|x-y-a_j(t-s)|^2}{N(t-s)}} e^{-\frac{|x-a_j s|^2}{N(1+s)}} dy \leq C(1+t)^{-\frac{1}{2}}(t-s)^{\frac{1}{2}}(1+s)^{\frac{1}{2}}e^{-\frac{|x-a_j t|^2}{N(1+t)}},
\notag
\ee
we have
\be
\begin{split}
I \leq C \sum_{j=1}^{n+1} (1+t)^{-\frac{1}{2}}e^{-\frac{|x-a_j t|^2}{M'(1+t)}}\int_0^t (t-s)^{-\frac{1}{2}}(1+s)^{-\frac{1}{2}} ds \leq C \sum_{j=1}^{n+1} (1+t)^{-\frac{1}{2}}e^{-\frac{|x-a_j t|^2}{M'(1+t)}}.
\end{split}
\notag
\ee
We now assume $j \neq k$. Noting first that for $j \neq k$,
\be
\int_{-\infty}^\infty e^{-\frac{|x-y-a_j(t-s)|^2}{N(t-s)}} e^{-\frac{|y-a_k s|^2}{N(1+s)}} dy \leq C(1+t)^{-\frac{1}{2}}(t-s)^{\frac{1}{2}}(1+s)^{-\frac{1}{2}}e^{-\frac{|x-a_j (t-s)-a_k|^2}{N(1+t)}},
\notag
\ee
we have
\be \label{NE1}
\begin{split}
I \leq C \sum_{j \neq k } (1+t)^{-\frac{1}{2}}\int_0^t  (t-s)^{-\frac{1}{2}}(1+s)^{-\frac{1}{2}} e^{-\frac{|x-a_j (t-s)-a_k s|^2}{N(1+t)}}ds =I'.
\end{split}
\ee
To estimate the right hand side($I'$) of \eqref{NE1}, we consider 6 cases only with assumption $x \leq 0$. The case $x\geq 0$ is entirely symmetric. \\

\textbf{Case 1.} $x \leq 0$ and  $0 \leq a_j <  a_k$. In this case, we can rewrite
\be \label{relation 1}
x-a_j (t-s)-a_k s =(x-a_jt)+(a_j-a_k)s.
\ee
Here,  $x-a_j t$ and $(a_j-a_k)s$ are both negative and there is no cancellation, so we have
\be \label{estimate II'}
\begin{split}
I
& \leq C \sum_{j=1}^{n+1} (1+t)^{-\frac{1}{2}}e^{-\frac{|x-a_j t|^2}{M'(1+t)}}\int_0^t  (t-s)^{-\frac{1}{2}}(1+s)^{-\frac{1}{2}} ds \\
&  \leq C \sum_{j=1}^{n+1} (1+t)^{-\frac{1}{2}}e^{-\frac{|x-a_j t|^2}{M'(1+t)}}.
\end{split}
\ee
\\

\textbf{Case 2.} $x \leq 0$ and  $0 \leq a_k < a_j$. This is exactly same as the case 1 with rewriting
\be \label{relation 2}
x-a_j (t-s)-a_k s = (x-a_k t)-(a_j-a_k)(t-s).
\ee
\\

\textbf{Case 3.} $x \leq 0$ and  $a_k < 0 \leq a_j$. In this case, we consider two subcases  $|x| \geq |a_k|t$ and  $|x| \leq |a_k|t$. For $|x| \geq |a_k|t$, $x-a_k t$ and $-(a_j-a_k)(t-s)$ are both negative and no cancellation occurs in \eqref{relation 2}, so we have the same estimate as \eqref{estimate II'}.
%\be
%\begin{split}
%|II'| \leq C \sum_{ k=1 }^{n+1} (1+t)^{-\frac{1}{2}}e^{-\frac{|x-a_k|^2}{M(1+t)}}\int_0^t  (t-s)^{-\frac{1}{2}}(1+s)^{-\frac{1}{2}} ds \leq  C \sum_{ k=1 }^{n+1}(1+t)^{-\frac{1}{2}}e^{-\frac{|x-a_k|^2}{M(1+t)}}.
%\end{split}
%\ee
In the event $|x|\leq  |a_k|t$, we integrate $I'$ separately $[0, t/2]$ and $[t/2, t]$. For $s \in [0, t/2]$, since $x-a_jt$ is negative and $(a_j-a_k)s$  is positive in \eqref{relation 1}, cancellation occurs. In this case, we use the following balance estimate: 
\be
\begin{split}
& (1+s)^{-\frac{1}{2}} e^{-\frac{|x-a_j (t-s)-a_k s|^2}{N(1+t)}} \\
& \qquad  \leq C\Big[(1+|x-a_jt|)^{-\frac{1}{2}} e^{-\frac{|x-a_j (t-s)-a_k s|^2}{N(1+t)}} + (1+s)^{-\frac{1}{2}} e^{-\frac{|x-a_jt|^2}{M'(1+t)}}\Big].
\notag
\end{split}
\ee
We can easily prove this by considering two cases $(a_j-a_k)s \geq C|x-a_j t|$ and $(a_j-a_k)s \leq C|x-a_j t|$ for some constant $C>0$ in the relation \eqref{relation 1}. So we have,
\be \label{balance estimate for s}
\begin{split}
%&  (1+t)^{-\frac{1}{2}}\int_0^{t/2}  (t-s)^{-\frac{1}{2}}(1+s)^{-\frac{1}{2}} e^{-\frac{|x-a_j (t-s)-a_k s|^2}{N(1+t)}}ds \\
I'
& \leq C(1+t)^{-\frac{1}{2}}(1+|x-a_jt|)^{-\frac{1}{2}}\int_0^{t/2}  (t-s)^{-\frac{1}{2}} e^{-\frac{|x-a_j (t-s)-a_k s|^2}{N(1+t)}}ds \\
& \qquad + (1+t)^{-\frac{1}{2}}e^{-\frac{|x-a_jt|^2}{M'(1+t)}}\int_0^{t/2} (t-s)^{-\frac{1}{2}}(1+s)^{-\frac{1}{2}} ds \\
& \leq C\Big[ (1+|x|+t)^{-\frac{1}{2}}(1+|x-a_jt|)^{-\frac{1}{2}} + (1+t)^{-\frac{1}{2}}e^{-\frac{|x-a_jt|^2}{M'(1+t)}}\Big].  \\
\end{split}
\ee
Here, the last inequality is from $|x|\leq  |a_k|t$. 
%Indeed,
%\be
%(1+t)^{-\frac{1}{2}} \leq C(1+t+t)^{-\frac{1}{2}}\leq C(1+|x|+t)^{-\frac{1}{2}}.
%\ee
In the case $s \in [t/2,t]$, we start with rewriting \eqref{relation 2}. Since $x-a_k t$ and $-(a_j-a_k)(t-s)$ have opposite signs in \eqref{relation 2}, we argue similarly the balance estimate for $\ds  (t-s)^{-\frac{1}{2}} e^{-\frac{|x-a_j (t-s)-a_k s|^2}{N(1+t)}}$. Thus we have
\be \label{balance estimate for t-s}
\begin{split}
%&  (1+t)^{-\frac{1}{2}}\int_{t/2}^t (1+s)^{-\frac{1}{2}} (t-s)^{-\frac{1}{2}} e^{-\frac{|x-a_j (t-s)-a_k s|^2}{N(1+t)}}ds \\
I'
& \leq C(1+t)^{-\frac{1}{2}}(1+|x-a_k t|)^{-\frac{1}{2}} \int_{t/2}^t  (1+s)^{-\frac{1}{2}} e^{-\frac{|x-a_j (t-s)-a_k s|^2}{N(1+t)}} ds\\
& \qquad + C(1+t)^{-\frac{1}{2}}e^{-\frac{|x-a_kt|^2}{M'(1+t)}}\int_{t/2}^t   (t-s)^{-\frac{1}{2}}(1+s)^{-\frac{1}{2}} ds \\
&  \leq C\Big[(1+|x|+t)^{-\frac{1}{2}}(1+|x-a_kt|)^{-\frac{1}{2}} + (1+t)^{-\frac{1}{2}} e^{-\frac{|x-a_kt|^2}{M'(1+t)}}\Big].
\end{split}
\ee
\\

\textbf{Case 4.} $x \leq 0$ and  $a_j < 0 \leq a_k$. This is exactly same as the case 3 by considering $|x|\geq |a_j|t$ and $|x|\leq |a_j|t$.
\\

\textbf{Case 5.} $x \leq 0$ and  $a_k < a_j<0 $. In this case, we consider 3 subcases, $|x|\geq |a_k|t$, $|x| \leq |a_j|t$ and $|a_j|t \leq |x| \leq |a_k|t$. For $|x|\geq |a_k|t$ and $|x| \leq |a_j|t$, we use \eqref{relation 2} and \eqref{relation 1} respectively because the expression $x-a_j(t-s)-a_k s$ has no cancellation.  In the event that $|a_j|t \leq |x| \leq |a_k|t$, we use the balance estimate for $s \in [0, t/2]$ and $s\in [t/2,t]$ similarly to \eqref{balance estimate for s} and \eqref{balance estimate for t-s}, respectively.
\\

\textbf{Case 6.} $x \leq 0$ and  $a_j < a_k<0 $. This is exactly same as the case 5 by considering $|x|\geq |a_j|t$, $|x| \leq |a_k|t$ and $|a_k|t \leq |x| \leq |a_j|t$.

\end{proof}

%%%%%%%%%%%%%%%%%%%%%%%%%%%%%%%%%%%%%%%%%%%%%%%%%%%%%%%%%%%%%%%%%%%%%%%%%%%%%%%%%%%%%%%%%%%%%%%
%%%%%%%%%%%%%%%%%%%%%%%%%%%%%%%%%%%%%%%%%%%%%%%%%%%%%%%%%%%%%%%%%%%%%%%%%%%%%%%%%%%%%%%%%%%%%%%
%%%%%%%%%%%%%%%%%%%%%%%%%%%%%%%%%%%%%%%%%%%%%%%%%%%%%%%%%%%%%%%%%%%%%%%%%%%%%%%%%%%%%%%%%%%%%%%

\begin{proof}[\textbf{Proof of the estimate (\ref{nonlinear estimate2})}] We now estimate the second nonlinear term of $v$,
\be
II = \int_0^t \int_{-\infty}^\infty |\tilde G_y(x,t-s;y)||\psi_1(y,s)|^2dyds \leq CE_0(\theta+\psi_1+\psi_2)(x,t).
\notag
\ee
Notice first that
\be
\begin{split}
II
& \leq  \int_0^t \int_{-\infty}^\infty (t-s)^{-1}\sum_{j=1}^{n+1}e^{-\frac{|x-y-a_j(t-s)|^2}{M(t-s)}} |\psi_1(y,s)|^2dyds \\
& \leq  \int_0^t \int_{-\infty}^\infty (t-s)^{-1}\sum_{j=1}^{n+1}e^{-\frac{|x-y-a_j(t-s)|^2}{M(t-s)}}\Big[\chi(y,s) \sum_{k=1}^{n+1}(1+|y|+s)^{-\frac{1}{2}}(1+|y-a_k s|)^{-\frac{1}{2}}\Big]^2 dyds . \\
\end{split}
\notag
\ee
It is enough to estimate
\be
\begin{split}
II' = \sum_{j,k=1}^{n+1}\int_0^t \int_{a_1s}^{a_{n+1}s} (t-s)^{-1}e^{-\frac{|x-y-a_j(t-s)|^2}{M(t-s)}}(1+|y|+s)^{-1}(1+|y-a_k s|)^{-1} dyds.   \\
\end{split}
\notag
\ee
We estimate $II'$ by considering three parts:   $x<a_1t$, $x>a_{n+1}t$ and $x \in [a_1t, a_{n+1}t]$. \\

For $x<a_1t$, we use  $x-y-a_j(t-s) = (x-a_1t)-(y-a_1s)-(a_j-a_1)(t-s)$ for $y \in [a_1s, a_js]$ and $x-y-a_j(t-s)=(x-a_1t)-(y-a_js)-(a_j-a_1)t $  for $y\in [a_js, a_{n+1}s]$ so that  they have no cancellation.  Thus,  we have
\be \label{NE2}
\begin{split}
& \int_0^t \int_{a_1s}^{a_{n+1}s} (t-s)^{-1}e^{-\frac{|x-y-a_j(t-s)|^2}{M(t-s)}}(1+|y|+s)^{-1}(1+|y-a_k s|)^{-1} dyds \\
& \leq e^{-\frac{|x-a_1t|^2}{2Mt}}\int_0^t \int_{a_1s}^{a_js} (t-s)^{-1}e^{-\frac{|x-y-a_j(t-s)|^2}{2M(t-s)}}(1+|y|+s)^{-1}(1+|y-a_k s|)^{-1} dyds \\
& \quad + e^{-\frac{|x-a_1t|^2}{2Mt}}\int_0^t \int_{a_js}^{a_{n+1}s} (t-s)^{-1}e^{-\frac{|x-y-a_j(t-s)|^2}{2M(t-s)}}(1+|y|+s)^{-1}(1+|y-a_k s|)^{-1} dyds \\
& \leq e^{-\frac{|x-a_1t|^2}{M't}}\int_0^t \int_{a_1s}^{a_{n+1}s} (t-s)^{-1}e^{-\frac{|x-y-a_j(t-s)|^2}{2M(t-s)}}(1+|y|+s)^{-1}(1+|y-a_k s|)^{-1} dyds \\
&  \leq (1+t)^{-\frac{1}{2}}e^{-\frac{|x-a_1t|^2}{M't}} \\
\end{split}
\ee
To argue the final inequality, let's show the following estimate with an assumption $a_1\leq 0$,  
\be \label{NE2.1}
\int_0^t \int_{a_1s}^{0}(t-s)^{-1}(1+|y|+s)^{-1}(1+|y-a_k s|)^{-\frac{1}{2}} e^{-\frac{|x-y-a_j(t-s)|^2}{2M(t-s)}}dyds  \\ \leq C(1+t)^{-\frac{1}{2}}. 
\ee
If $a_k >0$, then
\be
\begin{split}
& \int_0^t \int_{a_1s}^{0}(t-s)^{-1}(1+|y|+s)^{-1}(1+|y-a_k s|)^{-\frac{1}{2}} e^{-\frac{|x-y-a_j(t-s)|^2}{2M(t-s)}}dyds \\
& \leq \int_0^t \int_{a_1s}^{0} (t-s)^{-1}(1+|y|+s)^{-\frac{3}{2}} e^{-\frac{|x-y-a_j(t-s)|^2}{2M(t-s)}}dyds  \\
&  \leq C\int_0^{t/2} \int_{a_1s}^{0} (t-s)^{-1}(1+|y|+s)^{-\frac{3}{2}} dyds  \\
& \qquad \qquad + C \int_{t/2}^t  (t-s)^{-\frac{1}{2}}(1+s)^{-\frac{3}{2}} \int_{a_1s}^{0}  (t-s)^{-\frac{1}{2}}e^{-\frac{|x-y-a_j(t-s)|^2}{2M(t-s)}}dyds  \\
& \leq (1+t)^{-\frac{1}{2}}.
\end{split}
\notag
\ee
If $a_k \leq 0$, then
\be
\begin{split}
& \int_0^t \int_{a_1s}^{0}(t-s)^{-1}(1+|y|+s)^{-1}(1+|y-a_k s|)^{-\frac{1}{2}} e^{-\frac{|x-y-a_j(t-s)|^2}{bM(t-s)}}dyds \\
& \leq C \int_0^{t/2}  (t-s)^{-1}(1+s)^{-1} \int_{a_1s}^{0} (1+|y-a_ks|)^{-\frac{1}{2}}dy ds \\
& \qquad +\int_{t/2}^t  (t-s)^{-\frac{1}{2}}(1+s)^{-1} \int_0^{t/2}(t-s)^{-\frac{1}{2}}e^{-\frac{|x-y-a_j(t-s)|^2}{2M(t-s)}}  dyds \\
& \leq C \int_0^{t/2}  (t-s)^{-1}(1+s)^{-\frac{1}{2}}  ds + \int_{t/2}^t  (t-s)^{-\frac{1}{2}}(1+s)^{-1} ds \\
& \leq C(1+t)^{-\frac{1}{2}}.
\end{split}
\notag
\ee
Similarly,  we can prove 
\be
\int_0^t \int_{0}^{a_{n+1}s}(t-s)^{-1}(1+|y|+s)^{-1}(1+|y-a_k s|)^{-\frac{1}{2}} e^{-\frac{|x-y-a_j(t-s)|^2}{2M(t-s)}}dyds  \\ \leq C(1+t)^{-\frac{1}{2}}, 
\notag
\ee
with an assumption $a_{n+1}\geq 0$. \\

For $x>a_{n+1}t$, we use  $x-y-a_j(t-s) = (x-a_{n+1}t)-(y-a_js)+(a_{n+1}-a_j)t$  for $y \in [a_1s, a_js]$ and $x-y-a_j(t-s)=(x-a_{n+1}t)-(y-a_{n+1}s)-(a_j-a_{n+1})(t-s) $ and for $y\in [a_js, a_{n+1}s]$. Then we argue similarly to  estimate \eqref{NE2}. \\

%So we have
%\be
%\begin{split}
%& \int_0^t \int_{a_1s}^{a_{n+1}s} (t-s)^{-1}e^{-\frac{|x-y-a_j(t-s)|^2}{M(t-s)}}(1+|y|+s)^{-1}(1+|y-a_k s|)^{-1} dyds \\
%& \leq e^{-\frac{|x-a_{n+1}t|^2}{2Mt}}\int_0^t \int_{a_1s}^{a_js} (t-s)^{-1}e^{-\frac{|x-y-a_j(t-s)|^2}{M(t-s)}}(1+|y|+s)^{-1}(1+|y-a_k s|)^{-1} dyds \\
%& \quad + e^{-\frac{|x-a_{n+1}t|^2}{2Mt}}\int_0^t \int_{a_js}^{a_{n+1}s} (t-s)^{-1}e^{-\frac{|x-y-a_j(t-s)|^2}{M(t-s)}}(1+|y|+s)^{-1}(1+|y-a_k s|)^{-1} dyds \\
%& = e^{-\frac{|x-a_{n+1}t|^2}{2Mt}}\int_0^t \int_{a_1s}^{a_{n+1}s} (t-s)^{-1}e^{-\frac{|x-y-a_j(t-s)|^2}{2M(t-s)}}(1+|y|+s)^{-1}(1+|y-a_k s|)^{-1} dyds \\
%& \leq \theta^2 \quad \text{same as back.. }
%\end{split}
%\ee

We now assume $x \in [a_1t, a_{n+1}t]$ with $a_1<0$ and $a_{n+1}>0$. We estimate $II'$  into two parts, 
\be
II'_N=\int_0^t \int_{a_1s}^{0} \quad \text{and} \quad II'_P =  \int_0^t \int_{0}^{a_{n+1}s}. 
\notag
\ee
This is why we can assume  $a_1 \leq 0$ and $a_{n+1}\geq 0$. For $j=k$ which is a simple case, we first notice that
\be
\begin{split}
& (1+|y-a_j s|)^{-\frac{1}{2}} e^{-\frac{|x-y-a_j(t-s)|^2}{M(t-s)}} \\
& \leq (1+|x-a_jt|)^{-\frac{1}{2}}e^{-\frac{|x-y-a_j(t-s)|^2}{M(t-s)}} +(1+|y-a_j s|)^{-\frac{1}{2}}e^{-\frac{|x-a_jt|^2}{M't}}  e^{-\frac{|x-y-a_j(t-s)|^2}{bM(t-s)}},
\end{split}
\notag
\ee
for some constant $b>0$. Then we have
\be
\begin{split}
%& \int_0^t \int_{a_1s}^{0} (t-s)^{-1}e^{-\frac{|x-y-a_j(t-s)|^2}{M(t-s)}}(1+|y|+s)^{-1}(1+|y-a_j s|)^{-1} dyds  \\
%&  \leq \int_0^t \int_{a_1s}^{0} (t-s)^{-1}(1+|y|+s)^{-1}(1+|y-a_j s|)^{-\frac{1}{2}} (1+|x-a_jt|)^{-\frac{1}{2}}e^{-\frac{|x-y-a_j(t-s)|^2}{M(t-s)}}dyds  \\
%& \quad + \int_0^t \int_{a_1s}^{0} (t-s)^{-1}(1+|y|+s)^{-1}(1+|y-a_j s|)^{-1}e^{-\frac{|x-a_jt|^2}{M't}}  e^{-\frac{|x-y-a_j(t-s)|^2}{bM(t-s)}}dyds \\
II'_N 
&  \leq  C(1+|x-a_jt|)^{-\frac{1}{2}} \int_0^t \int_{a_1s}^{0}(t-s)^{-1}(1+|y|+s)^{-1}(1+|y-a_j s|)^{-\frac{1}{2}} e^{-\frac{|x-y-a_j(t-s)|^2}{M(t-s)}}dyds  \\
& \quad + Ce^{-\frac{|x-a_jt|^2}{M't}}  \int_0^t \int_{a_1s}^{0} (t-s)^{-1}(1+|y|+s)^{-1}(1+|y-a_j s|)^{-1} e^{-\frac{|x-y-a_j(t-s)|^2}{bM(t-s)}}dyds \\
& \leq  C\Big[(1+|x-a_jt|)^{-\frac{1}{2}}+e^{-\frac{|x-a_jt|^2}{M't}} \Big] \\
& \qquad \qquad \times \int_0^t \int_{a_1s}^{0}(t-s)^{-1}(1+|y|+s)^{-1}(1+|y-a_j s|)^{-\frac{1}{2}} e^{-\frac{|x-y-a_j(t-s)|^2}{bM(t-s)}}dyds  \\
& \leq  C\Big[(1+|x-a_jt|)^{-\frac{1}{2}}+e^{-\frac{|x-a_jt|^2}{M't}} \Big] (1+t)^{-\frac{1}{2}}. 
\end{split}
\notag
\ee
Here, we already proved the last inequality in \eqref{NE2.1}. Similarly,  we estimate $II'_P$.

Let's estimate $II'$ for  $j \neq  k$ with $a_k<0$(for the case of $a_k > 0$,  we can estimate $II'_p$ similarly to $II'_N$ in the case $a_k < 0$ and estimate $II'_N$ similarly to $II'_P$ in the case $a_k < 0$). It is easy to estimate $II'_p$ while we have to consider several cases again for $II'_N$ . For $II'_p$, since $a_k<0$ and $y \geq 0$, we say that $1+|y-a_k s| \sim 1+|y|+s$, and so we have
\be
\begin{split}
II'_p \leq  \sum_{j=1}^{n+1}\int_0^t \int_{0}^{a_{n+1}s} (t-s)^{-1}e^{-\frac{|x-y-a_j(t-s)|^2}{M(t-s)}}(1+|y|+s)^{-\frac{1}{2}}(1+s)^{-\frac{3}{2}}dyds.
\end{split}
\notag
\ee
Noting first that
\be
\begin{split}
& (1+|y|+s)^{-\frac{1}{2}}e^{-\frac{|x-y-a_j(t-s)|^2}{M(t-s)}} \\
& \leq (1+|x-a_j(t-s)|+s)^{-\frac{1}{2}}e^{-\frac{|x-y-a_j(t-s)|^2}{M(t-s)}}+(1+|y|+s)^{-\frac{1}{2}}e^{-\frac{|x-a_j(t-s)|^2}{M'(t-s)}}e^{-\frac{|x-y-a_j(t-s)|^2}{bM(t-s)}} ,
\end{split}
\notag
\ee
we have
\be
\begin{split}
II'_p
& \leq \int_0^t\int_{0}^{a_{n+1}s} (t-s)^{-1} (1+|x-a_j(t-s)|+s)^{-\frac{1}{2}}(1+s)^{-\frac{3}{2}}e^{-\frac{|x-y-a_j(t-s)|^2}{M(t-s)}}dyds \\
& \quad + \int_0^t \int_{0}^{a_{n+1}s} (t-s)^{-1}(1+|y|+s)^{-\frac{1}{2}}(1+s)^{-\frac{3}{2}}e^{-\frac{|x-a_j(t-s)|^2}{M'(t-s)}}e^{-\frac{|x-y-a_j(t-s)|^2}{bM(t-s)}}  dyds,\\
&  = A+B. 
\end{split}
\notag
\ee
For $B$,
\be
\begin{split}
B
& \leq\int_0^{t/2}  (t-s)^{-1} (1+s)^{-1}e^{-\frac{|x-a_j(t-s)|^2}{M'(t-s)}} ds +\int_{t/2}^t  (t-s)^{-\frac{1}{2}} (1+s)^{-2}e^{-\frac{|x-a_j(t-s)|^2}{M'(t-s)}} ds \\
& \leq   (1+t)^{-\frac{1}{2}}  \int_0^t (t-s)^{-\frac{1}{2}} (1+s)^{-\frac{1}{2}} e^{-\frac{|x-a_j(t-s)|^2}{M'(t-s)}} ds \\
& \leq CI'
\end{split}
\notag
\ee
which is estimated in the proof of \eqref{nonlinear estimate1}. 
For $A$, noting first that
\be
\begin{split}
(1+|x-a_j(t-s)|+s)^{-\frac{1}{2}}
&  \leq C(1+|x-a_jt|+s)^{-\frac{1}{2}} + C(1+|x-a_j(t-s)|+|x-a_jt|)^{-\frac{1}{2}} \\
& \leq C(1+|x-a_jt|)^{-\frac{1}{2}},
\end{split}
\notag
\ee
we have
\be
\begin{split}
A
&\leq  (1+|x-a_jt|)^{-\frac{1}{2}} \int_0^t \int_{0}^{a_{n+1}s}(t-s)^{-1} (1+s)^{-\frac{3}{2}}e^{-\frac{|x-y-a_j(t-s)|^2}{M(t-s)}}dyds \\
&\leq   (1+|x-a_jt|)^{-\frac{1}{2}} \int_0^t  (t-s)^{-\frac{1}{2}} (1+s)^{-\frac{3}{2}}ds \\
&\leq C(1+t)^{-\frac{1}{2}}(1+|x-a_jt|)^{-\frac{1}{2}} .
\end{split}
\notag
\ee

We now  estimate $II'_N$. To estimate this part, we agrue several cases. We try here only the case $x<0$ and $a_k<0<a_j$. We can agrue similarly other cases. Using $x-y-a_j(t-s)=(x-a_j(t-s)-a_ks)-(y-a_ks)$, we have

\be
\begin{split}
& (1+|y-a_k s|)^{-\frac{1}{2}}e^{-\frac{|x-y-a_j(t-s)|^2}{M(t-s)}} \\
& \quad \leq C(1+|x-a_j(t-s)-a_ks|)^{-\frac{1}{2}}e^{-\frac{|x-y-a_j(t-s)|^2}{M(t-s)}}  \\
& \qquad + (1+|y-a_k s|)^{-\frac{1}{2}}e^{-\frac{|x-a_j(t-s)-a_ks|^2}{N(t-s)}}e^{-\frac{|x-y-a_j(t-s)|^2}{bM(t-s)}}, 
\end{split}
\notag
\ee
for some constant $b>0$. 
Thus,
\be
\begin{split}
%& \int_0^t \int_{a_1s}^{0} (t-s)^{-1}e^{-\frac{|x-y-a_j(t-s)|^2}{M(t-s)}}(1+|y|+s)^{-1}(1+|y-a_k s|)^{-1} dyds  \\
II'_N
& \leq \int_0^t \int_{a_1s}^{0} (t-s)^{-1}(1+|y|+s)^{-1}(1+|y-a_k s|)^{-\frac{1}{2}} \\
& \qquad \qquad \qquad \qquad  \qquad \times (1+|x-a_j(t-s)-a_ks|)^{-\frac{1}{2}}e^{-\frac{|x-y-a_j(t-s)|^2}{M(t-s)}} dyds \\
& \qquad + \int_0^t \int_{a_1s}^{0} (t-s)^{-1}(1+|y|+s)^{-1}(1+|y-a_k s|)^{-1} \\
& \qquad \qquad \qquad \qquad \qquad \times e^{-\frac{|x-a_j(t-s)-a_ks|^2}{N(t-s)}}e^{-\frac{|x-y-a_j(t-s)|^2}{bM(t-s)}}dyds \\
& = A+B.
\end{split}
\notag
\ee
For $B$,
\be
\begin{split}
B
& \leq \int_0^{t/2} (t-s)^{-1}e^{-\frac{|x-a_j(t-s)-a_ks|^2}{N(t-s)}}\int_{a_1s}^{0} (1+|y|+s)^{-1}(1+|y-a_k s|)^{-1}dyds \\
& \qquad \qquad + \int_{t/2}^t (t-s)^{-1}e^{-\frac{|x-a_j(t-s)-a_ks|^2}{M(t-s)}}\int_{a_1s}^{0} (1+|y|+s)^{-1}e^{-\frac{|x-y-a_j(t-s)|^2}{bM(t-s)}}dyds \\
& \leq  C\int_0^{t/2} (t-s)^{-1}e^{-\frac{|x-a_j(t-s)-a_ks|^2}{M(t-s)}}(1+s)^{-1}\ln(1+s)ds \\
& \qquad  \qquad + \int_{t/2}^t (t-s)^{-\frac{1}{2}}e^{-\frac{|x-a_j(t-s)-a_ks|^2}{M(t-s)}}(1+s)^{-1}ds \\
& \leq C(1+t)^{-\frac{1}{2}}\int_0^t (t-s)^{-\frac{1}{2}}(1+s)^{-\frac{1}{2}}e^{-\frac{|x-a_j(t-s)-a_ks|^2}{M(t-s)}}ds  \\
& \leq C(\theta+\psi_1+\psi_2)(x,t).
\end{split}
\notag
\ee
%Here, for $[0,t/2]$, we integrate $(1+|y-a_k s|)^{-1}$ and ignore Gaussian kernel and for $[t/2,t]$, we integrate Gaussian kernel and ignore $(1+|y-a_k s|)^{-1}$. 
Here, the last inequality is proved in $I'$. 

To estimate $A$, we seperate $A$ into two parts $|x|\geq |a_k|t$ and $|x| \leq  |a_k|t$.  For $|x|\geq |a_k|t$, using
\be
x-a_j(t-s)-a_ks=(x-a_kt)-(a_j-a_k)(t-s),
\notag
\ee
for which there is  no cancellation, we have
\be
\begin{split}
A
& \leq (1+|x-a_kt|)^{-\frac{1}{2}} \int_0^t (t-s)^{-1}\int_{a_1s}^{0} (1+|y|+s)^{-1}(1+|y-a_k s|)^{-\frac{1}{2}} e^{-\frac{|x-y-a_j(t-s)|^2}{M(t-s)}} dyds\\
& \leq (1+|x-a_kt|)^{-\frac{1}{2}} \Big[ \int_0^{t/2}(t-s)^{-1}(1+s)^{-1}(1+s)^{\frac{1}{2}}ds+ \int_{t/2}^t (t-s)^{-1}(1+s)^{-1}(t-s)^{\frac{1}{2}}ds\Big]\\
%& \leq (1+t)^{-\frac{1}{2}} (1+|x-a_kt|)^{-\frac{1}{2}}  \int_{0}^t (t-s)^{-\frac{1}{2}}(1+s)^{-\frac{1}{2}}ds   \\
& \leq (1+t)^{-\frac{1}{2}} (1+|x-a_kt|)^{-\frac{1}{2}}.
\end{split}
\notag
\ee
For $|x| \leq  |a_k|t$, we divide again the analysis into the cases $s \in [0,t/2]$ and $s \in [t/2,t]$. In the case $s \in [0,t/2]$, using
\be
x-a_j(t-s)-a_ks=(x-a_jt)+(a_j-a_k)s,
\notag
\ee
%If $2|(a_j-a_k)s| \leq |x-a_jt|$,
%\be
%1+|x-a_j(t-s)-a_ks| = 1+|(x-a_jt)+(a_j-a_k)s| \geq 1+\frac{1}{2}|x-a_jt|=C(1+|x|+t)
%\ee
%If $2|(a_j-a_k)s| \geq |x-a_jt|$,
%\be
%1+|y|+s \geq C(1+|y|+|x|+t),
%\ee
we have
\be
\begin{split}
& (1+|x-a_j(t-s)-a_ks|)^{-\frac{1}{2}}(1+|y|+s)^{-\frac{1}{2}} \\
& \leq C[(1+|x-a_jt|)^{-\frac{1}{2}}(1+|y|+s)^{-\frac{1}{2}}+  (1+|x-a_j(t-s)-a_ks|)^{-\frac{1}{2}}(1+|y|+|x-a_jt|)^{-\frac{1}{2}}]. \\
%& \leq C[(1+|x|+t)^{-\frac{1}{2}}(1+|y|+s)^{-\frac{1}{2}} +  (1+|x-a_j(t-s)-a_ks|)^{-\frac{1}{2}}(1+|y|+|x|+t)^{-\frac{1}{2}} ].
\end{split}
\notag
\ee
%(we can easily prove this by cosidering $2|(a_j-a_k)s| \leq |x-a_jt|$ and $2|(a_j-a_k)s| \geq |x-a_jt|$)
Thus, we consider $A$ into two terms $A'$ and $A''$.
For $A'$,
\be
\begin{split}
A'
&  \leq   C(1+|x-a_jt|)^{-\frac{1}{2}}\int_0^{t/2}(t-s)^{-1}\int_{a_1s}^{0}(1+|y-a_k s|)^{-\frac{1}{2}} e^{-\frac{|x-y-a_j(t-s)|^2}{M(t-s)}} (1+|y|+s)^{-1} dyds \\
&  \leq C(1+|x-a_jt|)^{-\frac{1}{2}}\int_0^{t/2}(t-s)^{-1}(1+s)^{-1} \int_{a_1s}^{0}(1+|y-a_k s|)^{-\frac{1}{2}} e^{-\frac{|x-y-a_j(t-s)|^2}{M(t-s)}} dyds \\
&  \leq C(1+|x-a_jt|)^{-\frac{1}{2}}\int_0^{t/2}(t-s)^{-1}(1+s)^{-1}(1+s)^{\frac{1}{2}}ds \\
%&  \leq C(1+t)^{-\frac{1}{2}}(1+|x-a_jt|)^{-\frac{1}{2}}\int_0^{t/2}(t-s)^{-\frac{1}{2}}(1+s)^{-\frac{1}{2}}ds \\
&  \leq C(1+t)^{-\frac{1}{2}}(1+|x-a_jt|)^{-\frac{1}{2}}.
\end{split}
\notag
\ee
For $A''$,
\be 
\begin{split}
A''
%& = (1+|x-a_jt|)^{-\frac{1}{2}}\int_0^{t/2}(t-s)^{-1}(1+|x-a_j(t-s)-a_ks|)^{-\frac{1}{2}} \\
%& \qquad \qquad \qquad \qquad \times \int_{a_1s}^{0}(1+|y-a_k s|)^{-\frac{1}{2}} e^{-\frac{|x-y-a_j(t-s)|^2}{M(t-s)}} (1+|y|+s)^{-\frac{1}{2}} dyds\\
%& \leq C(1+|x-a_jt|)^{-\frac{1}{2}}\int_0^{t/2}(t-s)^{-1}(1+s)^{-\frac{1}{2}} \int_{a_1s}^{0}(1+|y-a_k s|)^{-\frac{1}{2}} (1+|x-a_j(t-s)-a_ks|)^{-\frac{1}{2}}dyds  \\
& \leq C(1+|x-a_jt|)^{-\frac{1}{2}}\int_0^{t/2}(t-s)^{-1}(1+s)^{-\frac{1}{2}} (1+|x-a_j(t-s)-a_ks|)^{-\frac{1}{2}}(1+s)^{\frac{1}{2}}ds  \\
& \leq  C(1+t)^{-1}(1+|x-a_jt|)^{-\frac{1}{2}}\int_0^{t/2} (1+|x-a_j(t-s)-a_ks|)^{-\frac{1}{2}}ds \\
%& \leq  C(1+t)^{-1}(1+|x-a_jt|)^{-\frac{1}{2}}((1+t)^{\frac{1}{2}}+(1+|x-a_jt|)^{\frac{1}{2}}) \\
%&  \leq C(1+t)^{-\frac{1}{2}}(1+|x-a_jt|)^{-\frac{1}{2}} +(1+t)^{-1} \\
%& \leq C(1+t)^{-\frac{1}{2}}(1+|x-a_jt|)^{-\frac{1}{2}} + (1+t)^{-\frac{1}{2}}(1+|x|+t)^{-\frac{1}{2}} \\
& \leq C(1+t)^{-1}(1+|x-a_jt|)^{-\frac{1}{2}}(1+t)^{\frac{1}{2}}\\
& \leq C(1+t)^{-\frac{1}{2}}(1+|x-a_jt|)^{-\frac{1}{2}}.
\end{split}
\notag
\ee
Here, the inequalities are from  $|x| \leq |a_k|t$ and $(|x|+t) \sim |x-a_jt|$ because of $x<0$ and $a_j>0$.

%\textbf{Claim.} $\ds\int_0^t (1+|x-a_j(t-s)-a_ks|)^{-\frac{1}{2}}ds \leq  C(1+t)^{\frac{1}{2}}+C(1+|x-a_jt|)^{\frac{1}{2}}$.
%\be
%\begin{split}
%& \int_0^t (1+|x-a_j(t-s)-a_ks|)^{-\frac{1}{2}}ds \\
%& =\int_{\frac{x-a_jt}{a_k-a_j}}^t (1+x-a_jt-(a_k-a_j)s)^{-\frac{1}{2}}ds +\int_0^{\frac{x-a_jt}{a_k-a_j}} (1-x+a_jt+(a_k-a_j)s)^{-\frac{1}{2}}ds \\
%& \leq C(1+x-a_kt)^{\frac{1}{2}}+ C(1-x+a_jt)^{\frac{1}{2}} \\
%& \leq C(1+t)^{\frac{1}{2}}+C(1+|x-a_jt|)^{\frac{1}{2}}.
%\end{split}
%\ee

In the case $s \in [t/2,t]$,  using
\be
x-a_j(t-s)-a_ks=(x-a_kt)-(a_j-a_k)(t-s), 
\notag
\ee
we have
\be \label{for contradiction}
\begin{split}
& (t-s)^{-\frac{1}{2}}(1+|x-a_j(t-s)-a_ks|)^{-\frac{1}{2}} \\
& \quad \leq  C[|x-a_kt|^{-\frac{1}{2}}(1+|x-a_j(t-s)-a_ks|)^{-\frac{1}{2}}+(t-s)^{-\frac{1}{2}}(1+|x-a_kt|)^{-\frac{1}{2}}].
\end{split}
\ee
Thus,
\be\label{NE5}
\begin{split}
A
& \leq C(1+t)^{-\frac{1}{2}}e^{-\frac{|x-a_kt|^2}{M't}} \\
& \quad +  \int_{t/2}^t \int_{a_1s}^{0} (t-s)^{-\frac{1}{2}}(1+|y|+s)^{-1}(1+|y-a_k s|)^{-\frac{1}{2}} \\
& \qquad \qquad \qquad \qquad \times e^{-\frac{|x-y-a_j(t-s)|^2}{M(t-s)}} |x-a_kt|^{-\frac{1}{2}}(1+|x-a_j(t-s)-a_ks|)^{-\frac{1}{2}}dyds \\
& \quad + \int_{t/2}^t \int_{a_1s}^{0} (t-s)^{-\frac{1}{2}}(1+|y|+s)^{-1}(1+|y-a_k s|)^{-\frac{1}{2}} \\
& \qquad \qquad \qquad \qquad \times e^{-\frac{|x-y-a_j(t-s)|^2}{M(t-s)}}(t-s)^{-\frac{1}{2}}(1+|x-a_kt|)^{-\frac{1}{2}}dyds \\
& =C(1+t)^{-\frac{1}{2}}e^{-\frac{|x-a_kt|^2}{M't}} + A' +A''.
\end{split}
\ee
Here, since $|A| \leq C(1+t)^{-\frac{1}{2}}$, for $|x-a_kt| \leq C\sqrt t$, we get the first term. So to estimate $A'$, we assume $|x-a_kt| \geq C\sqrt t $ and $t>1$. If $t\leq 1$, then $|x-a_kt| \geq C\sqrt t \geq Ct \geq C(t-s)$ which is a contraction to the expression \eqref{for contradiction}. Then,  we have
\be
\begin{split}
A'
& \leq |x-a_kt|^{-\frac{1}{2}}\int_{t/2}^t (1+s)^{-1}(1+|x-a_j(t-s)-a_ks|)^{-\frac{1}{2}}\int_{a_1s}^{0}  (t-s)^{-\frac{1}{2}}e^{-\frac{|x-y-a_j(t-s)|^2}{M(t-s)}}dyds \\
& \leq  (1+t)^{-1}(1+|x-a_kt|)^{-\frac{1}{2}}\int_{t/2}^t (1+|x-a_j(t-s)-a_ks|)^{-\frac{1}{2}}ds \\
& \leq  (1+t)^{-\frac{1}{2}}(1+|x-a_kt|)^{-\frac{1}{2}}
\notag
\end{split}
\ee
and
\be
\begin{split}
A''
& \leq (1+|x-a_kt|)^{-\frac{1}{2}}\int_{t/2}^t (1+s)^{-1}(t-s)^{-\frac{1}{2}}\int_{a_1s}^{0}  (t-s)^{-\frac{1}{2}}e^{-\frac{|x-y-a_j(t-s)|^2}{M(t-s)}}dyds \\
& \leq C(1+t)^{-\frac{1}{2}}(1+|x-a_kt|)^{-\frac{1}{2}}. 
\notag
\end{split}
\ee

\begin{remark}
We argue other cases very similarly. However, for some cases, we need to separate $A$ into three parts, not just two parts. For example, in the case of $a_k<a_j<0$, we need to consider $A$ by $|x|\geq |a_k|t$, $|x|\leq |a_j|t$ and $|a_j|t \leq |x| \leq |a_k|t$. 
\end{remark}

\end{proof}

%%%%%%%%%%%%%%%%%%%%%%%%%%%%%%%%%%%%%%%%%%%%%%%
%%%%%%%%%%%%%%%%%%%%%%%%%%%%%%%%%%%%%%%%%%%%%%%
%%%%%%%%%%%%%%%%%%%%%%%%%%%%%%%%%%%%%%%%%%%%%%%%

\begin{proof}[\textbf{Proof of the estimate (\ref{nonlinear estimate3})}] We now estimate the third nonlinear term of $v$,
\be
III = \int_0^t \int_{-\infty}^\infty |\tilde G_y(x,t-s;y)||\psi_2(y,s)|^2dyds \leq CE_0(\theta+\psi_1+\psi_2)(x,t).
\notag
\ee
Notice that
\be
\begin{split}
III
%& = \int_0^t \int_{-\infty}^\infty (t-s)^{-1}\sum_{j=1}^{n+1}e^{-\frac{|x-y-a_j(t-s)|^2}{M(t-s)}} \psi_2^2(y,s)dyds \\
& \leq  \int_0^t \int_{-\infty}^\infty (t-s)^{-1}\sum_{j=1}^{n+1}e^{-\frac{|x-y-a_j(t-s)|^2}{M(t-s)}} \\
& \quad  \times \Big[(1-\chi(y,s))(1+|y-a_1s|+\sqrt s)^{-\frac{3}{2}} +  (1-\chi(y,s))(1+|y-a_{n+1}s|+\sqrt s)^{-\frac{3}{2}}\Big]^2 dyds. \\
\end{split}
\notag
\ee
We here estimate
\be
\begin{split}
III'= \int_0^t \int_{-\infty}^{a_1s} (t-s)^{-1}e^{-\frac{|x-y-a_j(t-s)|^2}{M(t-s)}}(1+|y-a_1s|+\sqrt s)^{-3} dyds. \\
\end{split}
\notag
\ee
We can argue the other terms similarly. Using
\be
x-y-a_j(t-s)=(x-a_j(t-s)-a_1s)-(y-a_1s),
\notag
\ee
we have
\be
\begin{split}
& (1+|y-a_1s|+\sqrt s)^{-\frac{3}{2}} e^{-\frac{|x-y-a_j(t-s)|^2}{M(t-s)}} \\
& \quad \leq  (1+|(x-a_j(t-s)-a_1s)|+\sqrt s)^{-\frac{3}{2}} e^{-\frac{|x-y-a_j(t-s)|^2}{M(t-s)}} \\
& \qquad \qquad \qquad + (1+|y-a_1s|+\sqrt s)^{-\frac{3}{2}} e^{-\frac{|x-a_j(t-s)-a_1s|^2}{N(t-s)}}e^{-\frac{|x-y-a_j(t-s)|^2}{bM(t-s)}}, 
\notag
\end{split}
\ee
for some constant $b>0$. Thus,
\be\label{NE10}
\begin{split}
III'
%& = \int_0^t \int_{-\infty}^{a_1s} (t-s)^{-1}e^{-\frac{|x-y-a_j(t-s)|^2}{M(t-s)}}(1+|y-a_1s|+\sqrt s)^{-3} dyds \\
& \leq \int_0^t \int_{-\infty}^{a_1s} (t-s)^{-1}(1+|y-a_1s|+\sqrt s)^{-\frac{3}{2}} (1+|(x-a_j(t-s)-a_1s)|+\sqrt s)^{-\frac{3}{2}} e^{-\frac{|x-y-a_j(t-s)|^2}{M(t-s)}}dyds \\
& \quad + \int_0^t \int_{-\infty}^{a_1s} (t-s)^{-1}(1+|y-a_1s|+\sqrt s)^{-3}e^{-\frac{|x-a_j(t-s)-a_1s|^2}{N(t-s)}}e^{-\frac{|x-y-a_j(t-s)|^2}{bM(t-s)}}dyds. 
\end{split}
\ee
The second term of \eqref{NE10} is estimated by $I'$  because 
\be
\begin{split}
& \int_0^t \int_{-\infty}^{a_1s} (t-s)^{-1}(1+|y-a_1s|+\sqrt s)^{-3}e^{-\frac{|x-a_j(t-s)-a_1s|^2}{M'(t-s)}}e^{-\frac{|x-y-a_j(t-s)|^2}{bM(t-s)}}dyds \\
& \leq \int_0^{t/2} (t-s)^{-1} e^{-\frac{|x-a_j(t-s)-a_1s|^2}{N(t-s)}}\int_{-\infty}^{a_1s}(1+|y-a_1s|+\sqrt s)^{-3}dyds \\
& \qquad + \int_{t/2}^t  (t-s)^{-1}(1+\sqrt s)^{-3} e^{-\frac{|x-a_j(t-s)-a_1s|^2}{N(t-s)}}\int_{-\infty}^{a_1s}e^{-\frac{|x-y-a_j(t-s)|^2}{bM(t-s)}} dyds\\
& \leq C(1+t)^{-\frac{1}{2}} \int_0^{t/2} (t-s)^{-\frac{1}{2}} e^{-\frac{|x-a_j(t-s)-a_1s|^2}{N(t-s)}}(1+s)^{-1}ds\\
& \qquad + C(1+t)^{-\frac{1}{2}}\int_{t/2}^t  (t-s)^{-1}(1+s)^{-1} e^{-\frac{|x-a_j(t-s)-a_1s|^2}{N(t-s)}}(t-s)^{\frac{1}{2}}ds \\
& \leq CI'. 
\end{split}
\notag
\ee
We now prove the first term of $III'$:
\be
\begin{split}
III''
& =\int_0^t \int_{-\infty}^{a_1s} (t-s)^{-1}(1+|y-a_1s|+\sqrt s)^{-\frac{3}{2}} \\
& \qquad \qquad \qquad \qquad \times (1+|(x-a_j(t-s)-a_1s)|+\sqrt s)^{-\frac{3}{2}} e^{-\frac{|x-y-a_j(t-s)|^2}{M(t-s)}}dyds.  \\
\end{split}
\notag
\ee

For  $x < a_1 t$, using
\be
x-a_j(t-s)-a_1s = (x-a_1t)-(a_j-a_1)(t-s)
\notag
\ee
for which there is no cancellation, we have
\be
\begin{split}
III''
& \leq  \int_0^t \int_{-\infty}^{a_1s} (t-s)^{-1}(1+|y-a_1s|+\sqrt s)^{-\frac{3}{2}} (1+|x-a_1t|+\sqrt s)^{-\frac{3}{2}} e^{-\frac{|x-y-a_j(t-s)|^2}{M(t-s)}} dyds\\
& \leq C(1+t)^{-1}  (1+|x-a_1t|)^{-\frac{3}{2}} \int_0^{t/2} \int_{-\infty}^{a_1s}(1+|y-a_1s|+\sqrt s)^{-\frac{3}{2}} dyds \\
& \qquad +  (1+|x-a_1t|+\sqrt t)^{-\frac{3}{2}}(1+\sqrt t)^{-\frac{3}{2}}\int_{t/2}^t \int_{-\infty}^{a_1s} (t-s)^{-1} e^{-\frac{|x-y-a_j(t-s)|^2}{M(t-s)}} dyds\\
& \leq C(1+t)^{-1} (1+|x-a_1t|)^{-\frac{3}{2}} \int_0^{t/2}  (1+s)^{-\frac{1}{4}}ds \\
& \qquad +  (1+|x-a_1t|+\sqrt t)^{-\frac{3}{2}}(1+\sqrt t)^{-\frac{3}{2}}\int_{t/2}^t (t-s)^{-\frac{1}{2}} ds\\
%& \leq C(1+t)^{-1}  (1+|x-a_1t|)^{-\frac{3}{2}} \int_0^{t/2} (1+s)^{-\frac{1}{4}}ds \\
%& \qquad \qquad + (1+|x-a_1t|+\sqrt t)^{-\frac{3}{2}}(1+\sqrt t)^{-\frac{3}{2}}\int_{t/2}^t (t-s)^{-\frac{1}{2}}ds\\
& \leq C(1+t)^{-\frac{1}{4}}  (1+|x-a_1t|)^{-\frac{3}{2}} + C(1+t)^{-\frac{1}{4}}  (1+|x-a_1t|+\sqrt t)^{-\frac{3}{2}} \\
& \leq C(1+|x-a_1t|+\sqrt t)^{-\frac{3}{2}}  +(1+t)^{-\frac{1}{2}}e^{-\frac{|x-a_1t|^2}{M't}}.
\notag
\end{split}
\ee
For the last inequlaity, we consider two cases $|x-a_1t|\leq \sqrt t$ and $|x-a_1t| \geq \sqrt t$. Since $|III''| \leq (1+t)^{-\frac{1}{2}}$, we have $(1+t)^{-\frac{1}{2}}e^{-\frac{|x-a_1t|^2}{M't}}$ , for $|x-a_1t|\leq \sqrt t$. For  $|x-a_1t| \geq \sqrt t$,  $(1+|x-a_1t|)^{-\frac{3}{2}} \leq C(1+|x-a_1t|+\sqrt t)^{-\frac{3}{2}}$. \\
%For $|x|\leq |a_1|t$, we consider two cases $a_j<0$ and $a_j>0$. \\

For $x > a_{n+1} t$, using
\be
x-a_j(t-s)-a_1s = (x-a_{n+1}t)-(a_j-a_{n+1})(t-s)
\notag
\ee
for which there is no cancellation, we estimate $III''$ similarly to $x < a_1 t$. Thus, for  $x > a_{n+1} t$, we have 
\be
III'' \leq  C(1+|x-a_{n+1}t|+\sqrt t)^{-\frac{3}{2}}  +(1+t)^{-\frac{1}{2}}e^{-\frac{|x-a_{n+1}t|^2}{M't}}.
\notag
\ee
\\

Now we consider the last part $a_1t < x< a_{n+1}t$. We prove only the case $x < 0$ and $a_j\leq 0$. We can prove other cases similarly. In this part, we need to consider two cases. \\

\textbf{case1.} $|x| \leq |a_j|t$. Here, we have no cancellation in
\be
x-a_j(t-s)-a_1s=(x-a_jt)+(a_j-a_1)s.
\notag
\ee
Thus, we argue similarly to the event $x<a_1t$. \\

\textbf{case2.} $|x| \geq |a_j|t$.  In this case, we consider $s\in [0,t/2]$ and $s\in [t/2,t]$ similarly to the proof of \eqref{nonlinear estimate2}. For  $s\in [0,t/2]$, noting first that
\be
x-a_j(t-s)-a_1s=(x-a_jt)+(a_j-a_1)s,
\notag
\ee
we have
\be
\begin{split}
&(1+|x-a_j(t-s)-a_1s|+\sqrt s)^{-\frac{3}{2}} \\
& \quad \leq (1+|x-a_jt|+\sqrt s)^{-\frac{3}{2}} + (1+|x-a_j(t-s)-a_1s|+|x-a_jt|^{\frac{1}{2}}+\sqrt s)^{-\frac{3}{2}} .
\end{split}
\notag
\ee
Thus,
\be
\begin{split}
III''
& \leq \int_0^{t/2} \int_{-\infty}^{a_1s} (t-s)^{-1}(1+|y-a_1s|+\sqrt s)^{-\frac{3}{2}} (1+|x-a_jt|+\sqrt s)^{-\frac{3}{2}} e^{-\frac{|x-y-a_j(t-s)|^2}{M(t-s)}} dyds\\
& \qquad + \int_0^{t/2} \int_{-\infty}^{a_1s} (t-s)^{-1}(1+|y-a_1s|+\sqrt s)^{-\frac{3}{2}} \\
& \qquad \qquad \qquad \times (1+|x-a_j(t-s)-a_1s|+|x-a_jt|^{\frac{1}{2}}+\sqrt s)^{-\frac{3}{2}}e^{-\frac{|x-y-a_j(t-s)|^2}{M(t-s)}} dyds\\
& \leq  C\int_0^{t/2}  (t-s)^{-1}(1+|x-a_jt|+\sqrt s)^{-\frac{3}{2}} \int_{-\infty}^{a_1s}(1+|y-a_1s|+\sqrt s)^{-\frac{3}{2}}dyds\\
& \qquad +  \int_0^{t/2}(t-s)^{-1} (1+|x-a_jt|^{\frac{1}{2}}+\sqrt s)^{-\frac{3}{2}} \int_{-\infty}^{a_1s} (1+|y-a_1s|+\sqrt s)^{-\frac{3}{2}}dyds\\
& \leq C (1+|x-a_jt|)^{-\frac{1}{2}} \int_0^{t/2} (t-s)^{-1}(1+\sqrt s)^{-1} (1+\sqrt s)^{-\frac{1}{2}}ds\\
& \qquad + C(1+|x-a_jt|^{\frac{1}{2}})^{-1}\int_0^{t/2} (t-s)^{-1}(1+\sqrt s)^{-\frac{1}{2}}(1+\sqrt s)^{-\frac{1}{2}}ds \\
& \leq C(1+t)^{-\frac{1}{2}}(1+|x-a_jt|)^{-\frac{1}{2}} \\
& \leq C(1+t+|x|)^{-\frac{1}{2}}(1+|x-a_jt|)^{-\frac{1}{2}}.  
%& = A +B
\end{split}
\notag
\ee
%For $A$,
%\be
%\begin{split}
%A
%& \leq  C\int_0^{t/2}  (t-s)^{-1}(1+|x-a_jt|+\sqrt s)^{-\frac{3}{2}} \int_{-\infty}^{a_1s}(1+|y-a_1s|+\sqrt s)^{-\frac{3}{2}}dyds\\
%& \leq C (1+t)^{-1}(1+|x-a_jt|)^{-\frac{1}{2}} \int_0^{t/2} (1+\sqrt s)^{-1} (1+\sqrt s)^{-\frac{1}{2}}ds\\
%& \leq C(1+t)^{-\frac{1}{2}}(1+|x-a_jt|)^{-\frac{1}{2}} \\
%& \leq C(1+t+|x|)^{-\frac{1}{2}}(1+|x-a_jt|)^{-\frac{1}{2}} \quad (\text{because of $|x| \leq |a_1|t$})
%\end{split}
%\ee
%For $B$,
%\be
%\begin{split}
%B
%& \leq C(1+t)^{-1}(1+|x-a_jt|^{\frac{1}{2}})^{-1}\int_0^{t/2} (1+\sqrt s)^{-\frac{1}{2}}(1+\sqrt s)^{-\frac{1}{2}}ds \\
%& \leq C(1+t)^{-1}(1+|x-a_jt|^{\frac{1}{2}})^{-1}\int_0^{t/2} (1+\sqrt s)^{-\frac{1}{2}}(1+\sqrt s)^{-\frac{1}{2}}ds \\
%& \leq C(1+t)^{-\frac{1}{2}}(1+|x-a_jt|)^{-\frac{1}{2}} .
%\end{split}
%\ee
For $s\in [t/2,t]$, noting first that
\be
x-a_j(t-s)-a_1s=(x-a_1t)-(a_j-a_1)(t-s),
\notag
\ee
we have
\be
\begin{split}
&(t-s)^{-\frac{1}{2}}(1+|x-a_j(t-s)-a_1s|+\sqrt s)^{-\frac{3}{2}} \\
& \quad \leq |x-a_1t|^{-\frac{1}{2}}(1+|x-a_j(t-s)-a_1s|+\sqrt s)^{-\frac{3}{2}} + (t-s)^{-\frac{1}{2}} (1+|x-a_1t|+\sqrt s)^{-\frac{3}{2}} .
\end{split}
\notag
\ee
Thus,
\be
\begin{split}
III'' 
& \leq C(1+t)^{-\frac{1}{2}}e^{-\frac{|x-a_1t|^2}{Mt}} \\
& \quad +  \int_{t/2}^t  \int_{-\infty}^{a_1s} (t-s)^{-\frac{1}{2}}(1+|y-a_1s|+\sqrt s)^{-\frac{3}{2}}|x-a_1t|^{-\frac{1}{2}} \\
& \qquad \qquad \qquad \qquad \times  (1+|x-a_j(t-s)-a_1s|+\sqrt s)^{-\frac{3}{2}} e^{-\frac{|x-y-a_j(t-s)|^2}{M(t-s)}} dyds\\
& \quad +  \int_{t/2}^t \int_{-\infty}^{a_1s} (t-s)^{-1}(1+|y-a_1s|+\sqrt s)^{-\frac{3}{2}}  (1+|x-a_1t|+\sqrt s)^{-\frac{3}{2}} e^{-\frac{|x-y-a_j(t-s)|^2}{M(t-s)}} dyds\\
&  \leq C(1+t)^{-\frac{1}{2}}e^{-\frac{|x-a_1t|^2}{Mt}} +  A+B.
\notag
\end{split}
\ee
By the same argument as \eqref{NE5}, we assume $|x-a_1t| \geq C >0$ for $ A$. Finally, we have  
\be
\begin{split}
A
& \leq (1+|x-a_1t|)^{-\frac{1}{2}}\int_{t/2}^t  \int_{-\infty}^{a_1s} (t-s)^{-\frac{1}{2}}(1+|y-a_1s|+\sqrt s)^{-\frac{3}{2}} \\
& \qquad \qquad \times (1+|x-a_j(t-s)-a_1s|+\sqrt s)^{-\frac{3}{2}} e^{-\frac{|x-y-a_j(t-s)|^2}{M(t-s)}} dyds\\
& \leq (1+|x-a_1t|)^{-\frac{1}{2}}\int_{t/2}^t (1+\sqrt s)^{-\frac{3}{2}}(1+|x-a_j(t-s)-a_1s|+\sqrt s)^{-\frac{3}{2}} \\
& \qquad \qquad \times \int_{-\infty}^{a_1s} (t-s)^{-\frac{1}{2}}  e^{-\frac{|x-y-a_j(t-s)|^2}{M(t-s)}} dyds\\
& \leq (1+\sqrt t)^{-\frac{3}{2}}(1+|x-a_1t|)^{-\frac{1}{2}}\int_{t/2}^t (1+|x-a_j(t-s)-a_1s|+\sqrt s)^{-\frac{3}{2}}ds \\
& \leq C(1+t)^{-\frac{1}{2}}(1+|x-a_1t|)^{-\frac{1}{2}}
\notag
\end{split}
\ee
and 
\be
\begin{split}
B
& \leq \int_{t/2}^t \int_{-\infty}^{a_1s} (t-s)^{-1}(1+|y-a_1s|+\sqrt s)^{-\frac{3}{2}}  (1+|x-a_1t|+\sqrt s)^{-\frac{3}{2}} e^{-\frac{|x-y-a_j(t-s)|^2}{M(t-s)}} dyds\\
& \leq  (1+|x-a_1t|+\sqrt t)^{-\frac{3}{2}}\int_{t/2}^t  (t-s)^{-\frac{1}{2}}(1+\sqrt s)^{-\frac{3}{2}} ds \\
& \leq  (1+|x-a_1t|+\sqrt t)^{-\frac{3}{2}} \\
& \leq (1+t)^{-\frac{1}{2}}(1+|x-a_1t|)^{-\frac{1}{2}}.
\notag
\end{split}
\ee

\end{proof}

Now we complete the proof of \eqref{linear estimate} - \eqref{nonlinear estimate3}, which implies that we have \eqref{inequality of zeta}.

\end{proof}

\begin{proof}[\textbf{Proof of theorem (\ref{main theorem2})}]
By  Lemma \ref{main lemma}, we have  $\zeta(t) \leq C(E_0+\zeta^2(t))$ for all $t\geq 0$ for which $\zeta(t)$ defined in \eqref{zeta} is finite. Since $\zeta(t)$ is continous so long as it remains finite, it follows by continous induction that $\zeta(t) \leq 2CE_0$ for all $t\geq 0$ provided $E_0 \leq \frac{1}{4C^2}$ and (as holds without loss of generality) $C\geq 1$. Thus, recalling $\zeta(t):=\sup_{0\leq s\leq t, x\in \RR} |(v,\vp_t, \vp_x, \vp_{xx})(x,s)|(\theta + \psi_1+\psi_2)^{-1}$, we have
\be
|v(x,t)| \leq CE_0(\theta + \psi_1+\psi_2),
\notag
\ee
for all $t \leq 0$ and all $x\in \RR$.

\end{proof}

\medskip

{\bf Acknowledgement.}
This project was completed while studying within the PhD program
at Indiana University, Bloomington.
Thanks to my thesis advisor Kevin Zumbrun for suggesting
the problem and for helpful discussions.

\end{document}